\documentclass{cimart}

\DeclareMathOperator{\A}{A}
\DeclareMathOperator{\B}{B}
\DeclareMathOperator{\C}{C}
\DeclareMathOperator\Cend{Cend}
\DeclareMathOperator{\dcobound}{d}
\DeclareMathOperator\Der{Der}
\DeclareMathOperator{\Hom}{Hom}
\DeclareMathOperator{\Homol}{H}
\DeclareMathOperator{\im}{Im}
\DeclareMathOperator\Vir{Vir}
\newcommand\oo[1]{\mathbin {{}_{(#1)}}}

\title[Hochschild cohomology of universal associative conformal envelope of Virasoro]{Hochschild cohomology of the universal associative conformal envelope of the Virasoro Lie conformal algebra with coefficients in all finite modules}

\author{
    Hassan Alhussein, Pavel Kolesnikov and Viktor Lopatkin
    }

\authorinfo[
    Hassan Alhussein]{
     Siberian State University of Telecommunication and Informatics, Russia}{hassanalhussein2014@gmail.com }

\authorinfo[
   Pavel Kolesnikov]{
    Sobolev Institute of Mathematics, Russia}{%
    pavelsk77@gmail.com}

\authorinfo[
    Viktor Lopatkin]{
    HSE University, Russia}{%
    wickktor@gmail.com
    }

\abstract{%
    In this paper, we find the Hochschild cohomology groups of the universal associative conformal envelope $U(3)$ of the Virasoro Lie conformal algebra with respect to associative locality $N=3$ on the generator with coefficients in all finite modules. In order to obtain this result, we construct the Anick resolution via the algebraic discrete Morse theory and Gr\"obner--Shirshov basis.
    }



\VOLUME{33}
\YEAR{2025}
\ISSUE{3}
\NUMBER{7}
\DOI{https://doi.org/10.46298/cm.14674}

\begin{document}

\setcounter{tocdepth}{2}
\tableofcontents 

\section{Introduction}
Conformal algebra is introduced in \cite{KacValgBeginners} as an algebraic tool to study the singular part of the operator product expansion (OPE)
of chiral fields in 2-dimensional conformal field theory.

The study of universal structures for conformal algebras was initiated in \cite{Roitman99}. The
classical theory of finite-dimensional Lie algebras often needs universal constructions
like free algebras and universal enveloping algebras. This was a motivation for the
development of combinatorial issues in the theory of conformal algebras \cite{BF}.

For example, consider the Virasoro conformal algebra $\Vir$. One may fix a natural number $N$ 
and construct 
the associative conformal algebra $U(N)$ 
generated by a single 
element $v$ such that $(v\oo{n} v) = 0$ for $n\ge N$, and 
the commutation relations of $\Vir $ hold.
Obviously, $U(1)=0$; the algebra $U(2)$ 
is known as the Weyl conformal algebra 
(also denoted $\Cend_{1,x}$ \cite{BKL}). 
The structure of $U(3)=U(4)$ was studied in 
\cite{Kol2020-IAJC} by means of the Gr\"obner--Shirshov bases 
method.

For every conformal algebra $C$ one may construct 
an ``ordinary'' algebra $\mathcal A(C)$ called a 
coefficient algebra of~$C$, which inherits many properties of $C$ (associative, commutative, Lie, Jordan, etc.).
Every conformal module over $C$ is also a module over 
$\mathcal A(C)$. Moreover, it was proved in \cite{BKV} that the (reduced) cohomology of a conformal algebra $C$ may be calculated via the corresponding cochain complex of its coefficient $\mathcal A(C)$.

It was shown in \cite{Kozlov2017} that the second Hochschild 
cohomology groups $\Homol ^2(U(2), M)$ are trivial for every conformal 
(bi-)module~$M$, but for 
higher Hochschild cohomology the direct computation becomes too complicated.

In \cite{Akl} was applied discrete algebraic Morse theory to evaluate the dimensions 
    of the Hochschild cohomology groups with scalar coefficients
    of the universal associative conformal envelope 
    $U(3)$ of the Virasoro Lie conformal algebra relative to 
    the associative locality bound $N=3$ on the generator.

    The Hochschild cohomology groups of the 
Weyl associative conformal algebra $U(2)$ with coefficients in all finite modules was found in \cite{Ak}.

The motivation to study $U(3)$ is that if $M$ is a finite irreducible 
$\Vir $-module then $M$ is a finite irreducible 
$U(\Vir,N=3)$-module, $U(3)$ is a more adequate associative conformal 
envelope of $\Vir $ than the Weyl conformal algebra~$U(2)$.

In this paper we find the higher Hochschild 
cohomology groups $\Homol^n(U(3),M)$ of the universal associative conformal envelope 
    $U(3)$ of the Virasoro Lie conformal algebra with coefficients in all finite modules~$M$. In order to obtain this result we construct the Anick resolution for its coefficient algebra via the algebraic discrete Morse theory.

\section{Review of conformal algebras}
Throughout the paper, 
$\Bbbk$ is a field of characteristic zero and
$\mathbb Z_+$ is the set of nonnegative integers.
\subsection{Conformal algebra}
 A left $H=\Bbbk [\partial ]$-module is called a 
{\em conformal algebra} \cite{KacValgBeginners} if it is equipped with 
a linear operator $\partial : C\to C$
and with $\Bbbk$-bilinear product $a\oo n b$ such that the following axioms hold
\begin{enumerate}
    \item[(C1)] for every $a,b \in C$ there exists $N = N(a,b)\in \mathbb Z_+$ such that $(a \oo n b) = 0$ for all $n\geq N$,
    \item[(C2)] $(\partial a\oo n b) = -n(a \oo {n-1} b)$,
    \item[(C3)] $(a \oo n \partial b) = \partial(a \oo n b) + n(a \oo {n-1} b)$.
\end{enumerate}

The structure of a conformal algebra on an $H$-module $C$
may be expressed by means of
a single polynomial-valued map called $\lambda$-product: 
\[
\begin{gathered}
(\cdot \oo{\lambda} \cdot) : C \otimes C \longrightarrow C[\lambda],
\\
(a \oo{\lambda} b) =\sum_{n=0}^{N(a,b)-1} \frac{\lambda^{n}}{n!}(a \oo{n} b),
\end{gathered}
\] 
where $\lambda $ is a formal variable, 
satisfying the following axioms:
\begin{align}
    &(\partial a\oo\lambda b) = -\lambda (a\oo\lambda b), 
      \label{eq:3/2-lin(1)}\\
    &(a\oo\lambda \partial b) = (\partial+\lambda) (a\oo\lambda b).
      \label{eq:3/2-lin(2)}
\end{align}
The number $N=N(a,b)$ is said to be a {\em locality level} of $a,b\in C$.
\subsection{Formal power series and Conformal algebra}
The previous definition is just a characterization of the following construction. Let $A$ be an algebra. One can consider formal power series $a(z),b(z)\in A[[z,z^{-1}]]$ and define
their $n$-product by the formula \cite{KacForDistrib}:
\begin{equation}\label{eq:2.1}
(a \oo n b)(z)=\mathrm{Res}_{w}a(z)b(z)(w-z)^n;\quad n\ge0.
\end{equation}
where $\mathrm{Res}_w$ is the coefficient of $w^{-1}$
in the formal power series in $z$ and $w$. Then the locality condition
$(\mbox{C}1)$ is equivalent to the relation
\[
a(z)b(z)(w-z)^N=0
\]
and relation $(\mbox{C}2)$ holds for $\partial = \frac{d}{dz}$. In the case of Lie algebras, the family of $n$-products $a \oo n b$ describes
the singular part of the operator product expansion  of two local series $a(z)$ and $b(z)$.

\subsection{The coefficient algebra}

To any conformal algebra $C$ there corresponds
an ``usual'' algebra $A=\mathcal A(C)$, and it is called {\em the coefficient algebra} such that $C$ lies in $A[[z,z^{-1}]]$
as a subspace of local series with the $n$-th products given by \eqref{eq:2.1}, which is constructed in the following way.

Consider the space of Laurent series $C[t^{-1},t]$ in an independent variable $t$ with
coefficients in $C$. For $a\in C$, denote $a(n)=at^{n}$. As a linear space $A$ is
isomorphic to the quotient of $C[t^{-1},t]$ over the subspace generated by the vectors $(\partial a)(n)+na(n-1)$ for $a\in C$. The formula for the product in A is derived from
\[
a(n)\cdot b(m) = \sum\limits_{s\ge 0} \binom{n}{s} (a\oo{s} b) (n+m-s).
\]
Note that the sum here is finite due to $(\mbox{C}1)$. The space $\mathcal A_+(C)$ spanned by all $a(n)$, $n\in \mathbb Z_+$, $a\in C$,
is a subalgebra of $\mathcal A(C)$ which is closed under the derivation~$\partial $.

There is a correspondence between the identities in a conformal algebra
and the identities in its coefficient algebra. For example, $A(C)$ is associative if and only if 
\begin{equation}\label{eq:ConfAss}
a \oo n (b \oo m c) = \sum _{s \geq 0} \binom{n}{s}
(a \oo{n-s} b) \oo {m+s} c
\end{equation}
for all $a, b, c \in C, n,m \in \mathbb Z_+$. In terms of the $\lambda$-product, the last relation may be expressed by a single formula
\[
a \oo \lambda (b \oo \mu c) = (a \oo \lambda b) \oo {\lambda+\mu} c, \quad a, b, c \in C,
\]
where $\lambda$ and $\mu$ are independent commuting variables \cite{KacForDistrib}. As in the world of ordinary algebras, an associative conformal algebra $C$ turns
into a Lie one (denoted by $C^{(-)}$) relative to new operations
\[
 [a\oo\lambda b] = (a\oo\lambda b) - (b\oo{-\partial - \lambda } a),\quad a,b\in C.
\]
\begin{example}
The Lie conformal algebra with one generator $v$ as free $H$-module and $\lambda$-product
\[
 (v\oo\lambda v ) = (\partial +2\lambda )v
\]
is called the Virasoro conformal algebra. The coefficient algebra of $\Vir $ is the Witt algebra
 $\mathcal A(\Vir) = \Der \Bbbk [t,t^{-1}]$, 
 and $\mathcal A_+(\Vir )= \Der \Bbbk [t]$.
\end{example}
\subsection{Universal enveloping of  Virasoro Lie conformal algebras}
Given an integer $N\ge 2$, 
one may construct an associative conformal envelope
$U(N)$
for the Virasoro Lie conformal algebra $\Vir $ with 
a generator $v$ 
which is universal
in the class of all such envelopes $C$ that $N_C(v,v)\le N$.
For example, $U(2)= U(\Vir ; N=2)$ is the Weyl conformal 
algebra~$\Cend_{1,x}$; the structure of $U(3) = U(\Vir ; N=3)$ 
is more complicated, and it was studied in 
\cite{Kol_ProcICAC}. We have that 
$U(3) = U(\Vir ; N=3)$ is an associative conformal algebra generated by $v$ relative to
\[
 \deg_\lambda (v\oo\lambda v) <N=3,\quad
 (v\oo\lambda v) - (v\oo{-\partial-\lambda } v ) = (\partial+2\lambda )v.
\]
The structure of the universal enveloping 
associative conformal algebra $U(3)$ 
of the Virasoro Lie conformal algebra $\Vir $
relative to the locality bound $N=3$ 
may be expressed by means $(.(n).)$, 
is generated by a single element $v$
such that $v\oo{n} v = 0$ for $n\ge 3$. 
The remaining defining relation of $U(3)$ is 
\begin{equation}\label{eq:vir-comm}
2v\oo{1} v - \partial (v\oo{2} v)= 2v.
\end{equation}
\subsection{Conformal module}
A left conformal module $M$ over an associative conformal algebra $C$ is a $\Bbbk[\partial]$-module endowed with a $\Bbbk$-bilinear map $C\otimes M\longrightarrow M[\lambda]$ satisfying the following axioms 
\begin{gather}
    (\partial a\oo\lambda m) = -\lambda (a\oo\lambda m), 
      \quad 
    (a\oo\lambda \partial m) = (\partial+\lambda) (a\oo\lambda m),
      \label{eq:3/2-lin(mod1-2)}\\
    (a\oo{\lambda } (b\oo\mu m)) = ((a\oo{\lambda } b)\oo{\mu+\lambda } m),
\end{gather}
for $a,b\in C, v\in M$.

Similarly, a conformal action of a Lie conformal algebra $L$
on a module $M$ meets \eqref{eq:3/2-lin(mod1-2)} 
and the conformal analogue of the Jacobi identity:
\[
(a\oo\lambda (b\oo\mu m )) - (b\oo\mu (a\oo\lambda m)) =
((a\oo\lambda b)\oo{\lambda+\mu} m),
\]
for $a,b\in L$, $m\in M$.

\begin{example}
Given a 1-generated free $H$-module  $M = \Bbbk [\partial ]u$
and two scalars $\Delta, \alpha \in \Bbbk$, 
one may define conformal action of $\Vir $  on $M$ as
\begin{equation}\label{eq:Vir}
(v\oo\lambda u) = (\alpha +\partial +\Delta\lambda )u.
\end{equation}
Denote the conformal $\Vir$-module obtained by $M(\alpha, \Delta)$.
For $\Delta \ne 0$, this is an irreducible $\Vir $-module, 
and every finite irreducible $\Vir $-module is isomorphic 
to an appropriate $M_{(\alpha, \Delta)}$ \cite{ChengKac}.
\end{example}

If $M$ is a module over an (associative or Lie) conformal algebra 
$C$, then $M$ is also a module over the ordinary 
(associative or Lie, respectively)
algebra $\mathcal A_+(C)$. 
Namely, for $a\in C$, $n\in \mathbb Z_+$, $u\in M$ the element 
$a(n)u$ is the coefficient at $\lambda ^{n}/n!$ of $(a\oo\lambda u)$:
\[
a\oo\lambda u  = \sum\limits_{n\ge 0} \dfrac{\lambda^n}{n!} a(n)u.
\]

For every conformal $C$-module $M$ over an associative conformal algebra $C$, the space $M$ is also a $C^{(-)}$-module relative to the same conformal action. 
The converse construction has a restriction due to locality. 
For example, the module $M_{(\alpha,\Delta)}$
over the Virasoro (Lie) conformal algebra is also a module 
over its universal enveloping associative conformal algebra 
$U(2)$
if and only if $\Delta = 0$ or $\Delta =1$ {\em but} the module $M_{(\alpha,\Delta)}$
over the Virasoro (Lie) conformal algebra is also a module 
over its universal enveloping associative conformal algebra 
$U(3)$ for all $\Delta$.
\subsection{Hochschild cohomology for associative algebras} \label{2.6}
Let $\Lambda$ be an associative algebra with a unit over a field $\Bbbk$. Let us start with the bar resolution (see \cite{McLane}) 
$\mathsf{B}_\bullet=(\mathsf{B_n}, \mathrm d_n)$  
where
\[
\mathsf{B}_n:= \Lambda 
\otimes (\Lambda/\Bbbk)^{\otimes n}
\]
The basis of $\mathsf B_n$ 
as a $\Lambda$-left module 
is presented by 
$[a_1|a_2|\dots |a_n]$, where $a_i$ are nontrivial (i.e., not equal to 1) 
reduced normal forms relative to the fixed Gr\"obner--Shirshov basis.

The differential
$\mathrm d_n : \mathsf{B}_n
   \to \mathsf{B}_{n-1}$
is defined as follows:
\[
\mathrm{d}_n([a_1| \ldots | a_n]) = (a_1 \otimes 1) 
[a_2| \ldots |a_n]
+  \sum\limits_{i=1}^{n-1}(-1)^{i}   [a_1| \ldots | 
N(a_ia_{i+1})| \ldots | a_n].
\]
Suppose  $M$ is an ``ordinary'' $\Lambda $-left module over $\Lambda$. 
Denote
\[
\mathrm C_{\mathrm B}^n = \Hom _\Lambda (\mathrm B_n, M) 
\]
this is the space of Hochschild cochains. 

The Hochschild differential 
\[
\Delta_{\mathrm B}^n  : \mathrm C_{\mathrm B}^n \to \mathrm C_{\mathrm B}^{n+1} 
\]
is obtained through 
\[
\big (\Delta_{\mathrm B}^n\varphi\big )(\mathbf x) = \varphi \dcobound_{n+1}(\mathbf x), \quad \varphi\in \mathrm C_{\mathrm B}^{n},
\  \mathbf x\in \mathrm B_{n+1}.
\]
{\em The Hochschild cohomology} for associative algebra $\Lambda$ with values in $M$ is
\[
\Homol^n(\Lambda,M)=\ker\Delta_n/\im\Delta_{n-1}.
\]
\subsection{Hochschild cohomology for associative conformal algebras}\label{2.2} Let $C$ be an associative conformal algebra and $M$ is $\Lambda$-left conformal module.
The {\em basic Hochschild complex}  
$\tilde \C^\bullet (C,M)$ for associative conformal algebra $C$ (see \cite{BKV}) is cochain spaces $\tilde \C^n (C,M)$, $n=1,2,\ldots$,
each of them is the space of all maps 
\[
\varphi_{\bar \lambda }: 
C^{\otimes n}\to M[\bar \lambda ],
\]
where $\bar\lambda = (\lambda_1,\ldots, \lambda_{n})$,
satisfying the conformal anti-linearity condition: 
\[
 \varphi_{\bar \lambda }(a_1,\ldots,\partial a_i, \ldots , a_n) =
 -\lambda_i \varphi _{\bar \lambda }(a_1,\ldots, a_n),\quad i=1,\ldots, n.
\]
The Hochschild differential 
$\dcobound_n : \tilde \C^n(C,M) \to \tilde \C^{n+1}(C,M)$
on the basic complex is given by
\[
(\dcobound_n\varphi)_{\bar\lambda }
(a_1,\ldots, a_{n+1}) = 
a_1 \oo{\lambda_1} 
\varphi_{\bar\lambda_0} (a_2, \ldots , a_{n+1})  
 + \sum\limits_{i = 1}^{n} (-1)^i\varphi_{\bar\lambda_i}
  (a_1,\ldots , a_i \oo{\lambda_i} a_{i+1},\ldots , a_{n+1}),
\]
for 
$\bar \lambda = (\lambda_1, \ldots, \lambda_{n+1})$, 
$\bar\lambda_0 = (\lambda_2, \ldots, \lambda_{n+1})$, 
$\bar\lambda_i = (\lambda_1, \ldots, \lambda_i+\lambda_{i+1},
\ldots , \lambda_{n+1} )$, $i=1,\dots, n$. The cohomology of the  basic Hochschild complex is called {\em the basic Hochschild cohomology} $\tilde H^\bullet (C, M)$. 

For every $n\ge 1$, the cochain space $\tilde \C^n(C,M)$ 
is a left  $\Bbbk[\partial]$-module:
\[
(\partial\varphi)_{\bar \lambda}(a_1,\ldots,a_n)=(\partial+\sum_{i=1}^{n}\lambda_i)\varphi_{\bar \lambda}(a_1,\ldots,a_n).
\]
For every $n\ge 1$, the map $\dcobound_n$ commutes with $\partial$. 
The quotient complex
\[
\C^\bullet (C,M)= 
\tilde \C^\bullet (C, M)/\partial \tilde \C^\bullet (C,M)\]
is called the {\em reduced Hochschild complex} and its cohomology is called the {\em reduced  Hochschild cohomology} $H^\bullet (C, M)$.

Consider the ``ordinary'' Hochschild complex 
$\C^\bullet (\mathcal A_+(C),M)$.
The maps
\[
\partial_n^* : \C^n( \mathcal A_+(C),M) \to \C^n( \mathcal A_+(C) ,M)
\]
given by
\[
(\partial_n^* f)(\alpha_1,\ldots, \alpha_n)
= \partial f(\alpha_1,\ldots, \alpha_n)
- \sum\limits_{i=1}^n 
f(\alpha_1,\ldots, \partial\alpha_i, \ldots, \alpha_n),
\]
and $\partial a(n) = -n a(n-1)$ for $a\in C$, $n\ge 0$, 
commute with the Hochschild differentials.

In \cite{BKV} was proved that 
\begin{gather*}
\C^\bullet (C,M)\cong 
\C^\bullet ( \mathcal A_+(C), M)/\partial^*_\bullet \C^\bullet ( \mathcal A_+(C),M),\\
\tilde H^\bullet (C,M)= 
H^\bullet ( \mathcal A_+(C), M).
\end{gather*}

So we can calculate (reduced) Hochschild cohomology of 
a conformal algebra $C$ via the Hochschild complex of its coefficient algebra $\mathcal A_+(C)$ and its quotient. 
To do that, we construct the Anick resolution for $\mathcal A_+(C)$ by means of the Morse matching method which is explained in detail in \cite{Akl}. 
We will use the homotopy maps constructed in this way 
to transfer the map $\partial^*_\bullet $
to the dual complex obtained from the Anick resolution,
as explained in the following sections.  
\subsection{Anick resolution for associative algebras} \label{Anires.}
Let $\Lambda$ be an associative algebra with a unit over a field $\Bbbk$. The Anick resolution (see \cite{Anick1983})
$\mathsf{A}_\bullet=(\mathsf{A_n}, \mathrm \delta_n)$ is complex, 
where
\[
\mathsf{A}_n:= \Lambda 
\otimes \Bbbk V^{(n-1)}.
\]

The basis of $\mathsf A_n$ 
as a $\Lambda$-left module 
is presented by 
$V^{(n-1)}$. Following Anick \cite{Anick1983} the elements of  $V^{(n-1)}$ are defined as: the $V^{(1)}$  is the set of all leading terms  of relations from Gr\"obner--Shirshov basis of $\Lambda$. A word $v=x_{i_1}\ldots x_{i_t}$ is an $n$-{\em prechain} if and only if there exist $a_j,b_j \in \mathbb{Z}$, $1 \le j \le n$, satisfying the following conditions:
\begin{itemize}
    \item $1=a_1<a_2 \le b_1<\ldots<a_n \le b_{n-1}<b_n=t$;
    \item $ x_{i_{a_j}}\ldots x_{i_{b_j}} \in V$ for $1 \le j \le n$.
\end{itemize}
An $n$-prechain $x_{i_1}\ldots x_{i_t}$ is an {\em $n$-chain} 
if only if the integers $a_j,b_j$ can be chosen in such a way that  $x_{i_1}\ldots x_{i_t}$  is not an $m$-prechain for neither 
$s < b_m$, $1 \le m \le n$. The set of all $n$-chains is denoted $V^{(n)}$.

The Anick differential $\delta_n$ can be constructed  by means of the Morse matching method which is explained in detail in \cite{Akl} or via the following algorithm: define two  $\Lambda $-linear operators 
$\delta_n'$ and $\delta_n''$ whose recurrent 
application
leads us to the desired $\delta_n$.
\begin{enumerate}
\item Let $[w]\in V^{(n)}$, calculate the ordinary differential
\begin{multline*}
    \delta_{n+1}'[w] 
= w_1 [w_2|\ldots |w_n|w_{n+1}]
 +\sum\limits_{j=1}^n (-1)^j [w_1|w_2|\ldots |s_j |\ldots | w_{n+1}] \in \Lambda\otimes (\Lambda /\Bbbk)^{\otimes n},
\end{multline*}
where 
 $s_j$ are the $V$-reduced forms of the products $w_jw_{j+1}$, 
   $j=1,\dots, n$;
 \item If $[v]\in (\Lambda/\Bbbk)^{\otimes n}$ belongs to $V^{(n-1)}$, then
\[
 \delta_{n+1}''[v] = [v]
\]
and the computation is finished.
Otherwise, 
suppose $[v]=[v_1|\dots |v_n]$ does not belong to $V^{(n-1)} $ 
 (here all $v_k$'s are $V$-reduced)
 then there exists the largest integer 
 $i\ge 0$ such that 
 $v_1\dots v_i \in V^{(i-1)}$,
 $v_{i+1}$ may be presented as $v_{i+1}=v_{i+1}'v_{i+1}''$,
 and
   $[v_1|\ldots |v'_{i+1}]$ belongs to $V^{(i)}$.
Then set
\[
 \delta_{n+1}''[v] = (-1)^{i}\delta_{n+1}'([v_1|\ldots|v'_{i+1}|v''_{i+1}|\ldots|v_{n}])+[v].
\]
If such an index $i\ge 0$ does not exist then set
$\delta_{n+1}''([v])=0$.
\end{enumerate}

After finitely many steps, 
the computation of $\delta_{n+1}''$ finishes. Therefore, 
\[
 \delta_{n+1}(w) = (\delta_{n+1}'')^{k} \delta_{n+1}'[w], 
\]
where $k=k(w)\ge 1$ is a sufficiently large integer (see examples in Theorem \ref{Anickdiff}).

We have that Anick resolution is homotopy equivalent to bar resolution and homotopy maps $\mathrm f_n: \B_n \to \A_n$, $\mathrm g_n:\A_n \to \B_n$ are given (see \cite{Akl}) such that $\delta_n 
= \mathrm f_{n-1}\dcobound_n \mathrm g_n:
    \A_n\to \A_{n-1}$.
    
\section{Anick resolution for $\mathcal{A}_{+}(U(3))$}\label{sec:Morse}

\begin{definition}[\hspace{-.001cm}\cite{Roit2000}]
The algebra $\mathcal A_+(U(3))$ is generated by the elements $v(n)$, $n\ge 0$, 
relative to the following relations:
\begin{align}
 v(n)v(m) -3v(n-1)v(m+1) & \\\nonumber
 + \ 3v(n-2)v(m+2)& - v(n-3)v(m+3) = 0,\quad n\ge 3, \ m\ge 0, \\
 \label{eq:u3-defn-loc}
v(n)v(m)& - v(m)v(n) =  (n-m)v(n+m-1), \quad n>m\ge 0. 
\end{align}
\end{definition}

\begin{theorem}[\hspace{-.001cm}\cite{Akl}]\label{thm:GSB-U(3)} 
The Gr\"obner--Shirshov basis of $\mathcal A_+(U(3))$ 
consists of the relations
\begin{align}
v(1)v(0)&=v(0)v(1)+v(0),& \label{eq:u3-gsb-10} \\
v(n)v(m) &=\frac{nm}{n+m-1} v(1)v(n+m-1)-\frac{(n-1)(m-1)}{n+m-1}v(0)v(n+m) \\\nonumber
&+\frac{n(n-1)}{n+m-1}v(n+m-1),\quad n\ge 2. \label{eq:u3-gsb}
\end{align}
\end{theorem}
To find the Anick complex, we need two steps. First, we have to find the set of obstructions for $\mathcal A_+(U(3))$ relative to the given Gr\"obner-Shirshov basis (the set of leading terms in $\mathcal A_+(U(3))$) 
and the set of $n$-chains. Next, build a Morse graph and calculate the path weights for every $n$-chain $w$ and $(n-1)$-chain $w'$ for all $n\ge1$, all graphs were building in \cite{Akl} or via methods in Section \ref{Anires.}.

\begin{corollary}\label{cor:U(3)AnickChains}
The Anick chains $\Lambda^{(n-1)}$ 
are of the following form:
\[
 [v(m_1)|v(m_2) | \dots |v(m_{n-1})|v(m_n)], 
\]
where $m_1,\dots, m_{n-2}\ge 2$ and either $m_{n-1}\ge 2$
or $(m_{n-1}, m_n)=(1,0)$.
\end{corollary}

We will write $[m_1|\dots |m_n]$ instead of $[v(m_1)|\dots |v(m_n)]$ for the sake of simplicity.

\begin{theorem}\label{Anickdiff}
 The Anick differential $\delta_{n+1}: A_{n+1}\to A_{n}$ is given by the following rules: 
    \begin{align*}
    & \delta_{n+1}[v(i_1)v(i_2)\dots v(i_{n+1})]=v(i_1)[v(i_2)\dots v(i_{n+1})]\\
    &+\sum^n_{j=1}(-1)^{j}\frac{i_ji_{j+1}}{i_j+i_{j+1}-1}v(1)[v(i_1)v(i_2)\ldots v(i_j+i_{j+1}-1)\ldots v(i_{n+1})]\\
    &+\sum^n_{j=2}\sum^{j-1}_{t=1}(-1)^{j}\frac{i_ji_{j+1}}{i_j+i_{j+1}-1}(i_t-1)[v(i_1)v(i_2)\ldots v(i_j+i_{j+1}-1)\ldots v(i_{n+1})]\\
   & +\sum^n_{j=1}(-1)^{j}\frac{i_j(i_j-1)}{i_j+i_{j+1}-1}[v(i_1)v(i_2)\ldots v(i_j+i_{j+1}-1)\ldots v(i_{n+1})]\\
    &+\sum^n_{j=2}\sum^{j-1}_{t=1}(-1)^{j+1}i_t\frac{(i_j-1)(i_{j+1}-1)}{i_j+i_{j+1}-1}[v(i_1)\ldots v(i_t-1)\ldots v(i_j+i_{j+1})\ldots v(i_{n+1})]\\
   & +\sum^n_{j=1}(-1)^{j+1}\frac{(i_j-1)(i_{j+1}-1)}{i_j+i_{j+1}-1}v(0)[v(i_1)\ldots v(i_j+i_{j+1})\ldots v(i_{n+1})],
    \end{align*}
    
    where $i_1,i_2,\ldots,i_n \ge 2$, $i_{n+1}\ge 0$, and 
    
    \begin{align*}
    &\delta_{n+1}[v(i_1)v(i_2)\ldots v(i_{n-1})v(1)v(0)]=v(i_1)[v(i_2)\ldots v(i_{n-1})v(1)v(0)]\\
    \end{align*}
    \begin{align*}
    &+\sum^{n-2}_{j=1}(-1)^{j}\frac{i_ji_{j+1}}{i_j+i_{j+1}-1}v(1)[v(i_1)\ldots v(i_j+i_{j+1}-1)\ldots v(i_{n-1})v(1)v(0)]\\
    &+\sum^{n-2}_{j=1}(-1)^{j+1}\frac{(i_j-1)(i_{j+1}-1)}{i_j+i_{j+1}-1}v(0)[v(i_1)\ldots v(i_j+i_{j+1})\ldots v(i_{n-1})v(1)v(0)]\\
    &+\sum^{n-2}_{j=1}(-1)^{j}\frac{i_j(i_j-1)}{i_j+i_{j+1}-1}[v(i_1)\ldots v(i_j+i_{j+1}-1)\ldots v(i_{n-1})v(1)v(0)]\\
    &+\sum^{n-2}_{j=2}\sum^{j-1}_{t=1}(-1)^{j+1}i_t\frac{(i_j-1)(i_{j+1}-1)}{i_j+i_{j+1}-1}[v(i_1)\ldots v(i_t-1)\ldots  v(i_j+i_{j+1})\ldots v(i_{n-1})v(1)v(0)]\\
    &+\sum^{n-2}_{j=2}\sum^{j-1}_{t=1}(-1)^{j}(i_t-1)\frac{i_ji_{j+1}}{i_j+i_{j+1}-1}[v(i_1)\ldots v(i_j+i_{j+1}-1)\ldots v(i_{n-1})v(1)v(0)]\\
    &+\sum^{n-1}_{j=1}(-1)^{j}i_j[v(i_1)\ldots v(i_j-1)\ldots v(i_{n-1})v(1)]+\sum^{n-1}_{j=1}(-1)^{j}(i_j-1)[v(i_1)\ldots v(i_{n-1})v(0)]\\
    &+(-1)^n[v(i_1)v(i_2)\ldots v(i_{n-1})v(0)]+(-1)^nv(0)[v(i_1)v(i_2)\ldots v(i_{n-1})v(1)],
    \end{align*}
    where $i_1,i_2\ldots,i_{n-1}\ge2$.
    \end{theorem}
    \begin{proof}
We can 
evaluate $\delta_{n+1}: \A_{n+1}\to \A_{n}$ by means of the Morse graph as a similar way which is shown in \cite{Akl} (see Figure 1-4) considering that in  \cite{Akl} was founded  $\delta_{n+1}$ for trivial module or via methods explained in Section \ref{Anires.}.    
    
For $n=1$ we have the following set of obstructions
\[
V^{(1)}=\{v(n)v(m), v(1)v(0); n\ge2,m\ge0\}.
\]
Let us  
evaluate $\delta_2: A_{2}\to A_{1}$ by means of the methods explained in Section \ref{Anires.}. Suppose 
$[w] = [v(n)|v(m)]$.
Then calculate 
\[
\begin{aligned}
\delta_2'[n|m] =& \ v(n)[v(m)] - \dfrac{nm}{n+m-1} [v(1)v(n+m-1)] + \dfrac{(n-1)(m-1)}{n+m-1} [v(0)v(n+m)]\\
 -& \ \dfrac{n(n-1)}{n+m-1} [v(n+m-1)];\quad n\ge2, m\ge0.\\
 \end{aligned}
 \]
 Apply $\delta_2''$ to each summand:
 \begin{align*}
\delta_2''([v(1)v(n+m-1)]) = & \ \delta_2' ([v(1)|v(n+m-1)]) + [v(1)v(n+m-1)]\\
= & \ v(1)[v(n+m-1)], \quad i=0,\\
\delta_2''([v(0)v(n+m)]) = & \ \delta_2' ([v(0)|v(n+m)]) + [v(0)v(n+m)]
=v(0)[v(n+m)], \quad i=0,\\
\end{align*}
\begin{align*}
\delta_2[n|m] = & \ v(n)[v(m)] - \dfrac{nm}{n+m-1} v(1) [v(n+m-1)]\\
 + & \ \dfrac{(n-1)(m-1)}{n+m-1} v(0)[v(n+m)]\\
 - & \ \dfrac{n(n-1)}{n+m-1} [v(n+m-1)];\quad n\ge2, m\ge0.
 \end{align*}
 Similar way
 \[
\delta_2[1|0] = v(1)[v(0)] -v(0)[v(1)] - [v(0)].
\]
 For $n=2$ we have
\[
V^{(2)}=\left\{v(n)v(m)v(p), v(n)v(1)v(0); n,m\ge2,p\ge0\right\}.
\]

Suppose 
$[w] = [v(n)|v(1)|v(0)]$.
Then calculate 
\begin{align*}
\delta_3'([v(n)|v(1)|v(0)]) =& \ v(n)[v(1)|v(0)]-[v(1)v(n)|v(0)]\\
-& \ (n-1)[v(n)v(0)]+[v(n)|v(0)v(1)] +[v(n)v(0)].
\end{align*}
Since $[v(1)|v(0)],[v(n)v(0)]\in V^{(1)}$, it is enough 
to calculate 
$(\delta_3'')^k([v(1)v(n)|v(0)])$ and $  (\delta_3'')^k([v(n)|v(0)v(1)]) $ for $k=1,2,\dots $
Here
\begin{align*}
\delta_3''([v(n)|v(0)v(1)])
=& \ -\delta_3'([v(n)|v(0)|v(1)]) + [v(n)|v(0)v(1)] \\
=& \ [v(0)v(n)|v(1)]+n[v(n-1)v(1)], 
\quad i=1.
\end{align*}
Note that 
$[v(n-1)v(1)]\in V^{(1)}$ and calculate 
\begin{align*}
\delta_3''([v(0)v(n)|v(1)])
= & \ \delta_3'([v(0)|v(n)|v(1)])+[v(0)v(n)|v(1)]
=v(0)[v(n)|v(1)],\\
\delta_3''([v(1)v(n)|v(0)]) = & \ \delta_3'([v(1)|v(n)|v(0)])+[v(1)v(n)|v(0)]
=v(1)[v(n)|v(0)],\\
\delta_3[v(n)v(1)v(0)] = & \ v(n)[v(1)v(0)]+n[v(n-1)v(1)]
-v(1)[v(n)v(0)]\\
-& \ (n-2)[v(n)v(0)]+v(0)[v(n)v(1)];\quad n\ge2.
\end{align*}

    Similar way for
    \[
    \begin{aligned}
    \delta_3[v(n)v(m)v(p)]&{}=v(n)[v(m)v(p)]-\frac{nm}{n+m-1}v(1)[v(n+m-1)v(p)]\\
    &+\frac{(n-1)(m-1)}{n+m-1}v(0)[v(n+m)v(p)]-\frac{n(n-1)}{n+m-1}[v(n+m-1)v(p)]\\
    &+\frac{mp}{m+p-1}v(1)[v(n)v(m+p-1)]+\frac{(n-1)mp}{m+p-1}[v(n)v(m+p-1)]\\
    &-\frac{(m-1)(p-1)}{m+p-1}v(0)[v(n)v(m+p)]\\
    &-\frac{n(m-1)(p-1)}{m+p-1}[v(n-1)v(m+p)]\\
    &+\frac{m(m-1)}{m+p-1}[v(n)v(m+p-1)];\quad n,m\ge2, p\ge0.
    \end{aligned}
    \]
   And also for $\delta_{n+1}$. Suppose 
$[w] = [v(i_1)v(i_2)\dots v(i_{n+1})]$; $i_1,\ldots,i_n\ge2, i_{n+1}\ge0$.
Then calculate 
\[
    \begin{aligned}
    \delta'_{n+1}[v(i_1)|v(i_2)|\dots |v(i_{n+1})]&=v(i_1)[v(i_2)|\dots |v(i_{n+1})]\\
    &+\sum^n_{j=1}(-1)^{j}\frac{i_ji_{j+1}}{i_j+i_{j+1}-1}\\
    &\times[v(i_1)|v(i_2)|\ldots|v(1) v(i_j+i_{j+1}-1)|\ldots |v(i_{n+1})]\\
   & +\sum^n_{j=1}(-1)^{j}\frac{i_j(i_j-1)}{i_j+i_{j+1}-1}\\
   &\times[v(i_1)v(i_2)\ldots v(i_j+i_{j+1}-1)\ldots v(i_{n+1})]\\
   & +\sum^n_{j=1}(-1)^{j+1}\frac{(i_j-1)(i_{j+1}-1)}{i_j+i_{j+1}-1}\\
   &\times[v(i_1)|\ldots |v(0)v(i_j+i_{j+1})|\ldots |v(i_{n+1})].
    \end{aligned}
    \]
    Since $[v(i_1)v(i_2)\ldots v(i_j+i_{j+1}-1)\ldots v(i_{n+1})]\in V^{(n)}$, it is enough 
to calculate 
\[
(\delta_{n+1}'')^k([v(i_1)|v(i_2)|\ldots|v(1) v(i_j+i_{j+1}-1)|\ldots |v(i_{n+1})]), \ \mbox{and} 
\]
\[
(\delta_{n+1}'')^k([v(i_1)|\ldots |v(0)v(i_j+i_{j+1})|\ldots |v(i_{n+1})]), 
\]
for $k=1,2,\dots$
Here
\[
\begin{aligned}
& \ \delta_{n+1}''([v(i_1)|v(i_2)|\ldots|v(1) v(i_j+i_{j+1}-1)|\ldots |v(i_{n+1})])\\
=& \ \delta_{n+1}'[v(i_1)|v(i_2)|\ldots|v(1)| v(i_j+i_{j+1}-1)|\ldots |v(i_{n+1})]\\
+& \ (-1)^j[v(i_1)|v(i_2)|\ldots|v(1) v(i_j+i_{j+1}-1)|\ldots |v(i_{n+1})]\\
+& \ (-1)^{j-1}(i_{j-1}-1)[v(i_1)|v(i_2)|\ldots|v(i_{j-1})| v(i_j+i_{j+1}-1)|\ldots |v(i_{n+1})].\\
 \end{aligned}
\]
By repeating this step,
\[
\begin{aligned}
& \ (\delta_{n+1}'')^{j-1}([v(i_1)|v(i_2)|\ldots|v(1) v(i_j+i_{j+1}-1)|\ldots |v(i_{n+1})])\\
=& \ v(1)[v(i_1)|v(i_2)|\ldots|v(i_{j-1})|v(i_j+i_{j+1}-1)|\ldots |v(i_{n+1})]\\
+& \ \sum^{j-1}_{t=1}(-1)^{j}\frac{i_ji_{j+1}}{i_j+i_{j+1}-1}(i_t-1)[v(i_1)v(i_2)\ldots v(i_j+i_{j+1}-1)\ldots v(i_{n+1})].
\end{aligned}
\]
Similar way 
\[
\begin{aligned}
&(\delta_{n+1}'')^{j-1}([v(i_1)|v(i_2)|\ldots|v(0) v(i_j+i_{j+1}-1)|\ldots |v(i_{n+1})])\\
=& \ v(0)[v(i_1)|v(i_2)|\ldots|v(i_{j-1})|v(i_j+i_{j+1})|\ldots |v(i_{n+1})]\\
+& \ \sum^{j-1}_{t=1}(-1)^{j+1}i_t[v(i_1)\ldots v(i_t-1)\ldots v(i_j+i_{j+1})\ldots v(i_{n+1})].
\end{aligned}
\]
Therefore
\[
    \begin{aligned}
    \delta_{n+1}[v(i_1)v(i_2)\dots v(i_{n+1})]&=v(i_1)[v(i_2)\dots v(i_{n+1})]\\
    &+\sum^n_{j=1}(-1)^{j}\frac{i_ji_{j+1}}{i_j+i_{j+1}-1}v(1)\\
    &\times[v(i_1)v(i_2)\ldots v(i_j+i_{j+1}-1)\ldots v(i_{n+1})]\\
    &+\sum^n_{j=2}\sum^{j-1}_{t=1}(-1)^{j}\frac{i_ji_{j+1}}{i_j+i_{j+1}-1}(i_t-1)\\
    &\times[v(i_1)v(i_2)\ldots v(i_j+i_{j+1}-1)\ldots v(i_{n+1})]\\
   & +\sum^n_{j=1}(-1)^{j}\frac{i_j(i_j-1)}{i_j+i_{j+1}-1}[v(i_1)v(i_2)\ldots v(i_j+i_{j+1}-1)\ldots v(i_{n+1})]\\
    &+\sum^n_{j=2}\sum^{j-1}_{t=1}(-1)^{j+1}i_t\frac{(i_j-1)(i_{j+1}-1)}{i_j+i_{j+1}-1}\\
    &\times[v(i_1)\ldots v(i_t-1)\ldots v(i_j+i_{j+1})\ldots v(i_{n+1})]\\
   & +\sum^n_{j=1}(-1)^{j+1}\frac{(i_j-1)(i_{j+1}-1)}{i_j+i_{j+1}-1}v(0)\\
   &\times[v(i_1)\ldots v(i_j+i_{j+1})\ldots v(i_{n+1})].
    \end{aligned}
    \]
    By this method one can find $\delta_{n+1}[v(i_1)v(i_2)\dots v(i_{n-1})v(1)v(0)]$.
\end{proof}
    \subsection{Anick resolution for differential algebras}

Let $C$ be an associative conformal algebra. During the whole of the rest of the paper 
$\Lambda $ represents for 
the augmented algebra  
$\Lambda = \mathcal A_+ \oplus \Bbbk 1$,
where 
$\mathcal A_+ = \mathcal A_+(C)$, 
and the augmentation is given by
$\varepsilon (\mathcal A_+)=0$.
Then $\mathcal A_+=\Lambda /\Bbbk $, 
and
$\Lambda $ has a derivation $\partial $ such that
$\partial (a(n)) = -na(n-1)$, $\partial (1)=0$.

For every $n\ge 1$, there is a $\Bbbk$-linear map 
\[
\partial_n: \mathrm B_n\to \mathrm B_n,
\]
\[
\partial_n(\beta [\alpha_1|\dots |\alpha_n])
= \partial(\beta) [\alpha_1|\dots |\alpha_n]
+\sum\limits_{i=1}^n 
 \beta [\alpha_1 | \dots |\partial(\alpha_i)| \dots |\alpha_n],
\]
$\alpha_i\in \mathcal A_+$, $\beta \in \Lambda $. 
This is not a morphism of complexes, 
but it causes a morphism of dual cochain complexes
which can be assigned to the Anick resolution 
via homotopy.
That is to say, suppose $M$ is a left conformal $C$-module. Thus, 
$M$ is an ``ordinary'' $\Lambda $-module. 
Denote
\[
\mathrm C_{\mathrm B}^n = \Hom _\Lambda (\mathrm B_n, M) \simeq 
\Hom_\Bbbk (A^{\otimes n}, M), 
\]
this is the space of Hochschild cochains. 
The Hochschild differential 
\[
\Delta_{\mathrm B}^n  : \mathrm C_{\mathrm B}^n \to \mathrm C_{\mathrm B}^{n+1} 
\]
is obtained through 
\[
\big (\Delta_{\mathrm B}^n\varphi\big )(\mathbf x) = \varphi \dcobound_{n+1}(\mathbf x), \quad \varphi\in \mathrm C_{\mathrm B}^{n},
\  \mathbf x\in \mathrm B_{n+1}.
\]

Note that the $\Lambda $-module $M$ is endowed
 with a derivation 
also denoted by $\partial$ (the same as in the definition of a conformal module), so that 
$\partial (a(n)u) = -na(n-1)u +a(n)\partial(u)$, for $a\in C$, $u\in M$, 
$n\in \mathbb Z_+$.

Then for every $n\ge 1$ the map
\[
D_{\mathrm B}^n : \mathrm C_{\mathrm B}^n \to \mathrm C_{\mathrm B}^n 
\]
is expressed as
\[
(D_{\mathrm B}^n \varphi)(\mathbf x) = \partial(\varphi (\mathbf x))
-\varphi (\partial_n(\mathbf x)),
\quad 
\varphi\in \mathrm C_{\mathrm B}^{n},
\  \mathbf x\in \mathrm B_{n+1},
\]
and it is a morphism of complexes:
\[
D_{\mathrm B}^{n+1}\Delta^n_{\mathrm B} 
=  \Delta^n_{\mathrm B} D_{\mathrm B}^{n}.
\]

Let us convert the mapping $D^\bullet _{\mathrm B}$ 
from 
$\mathrm C_{\mathrm B}^\bullet $
to the complex 
$\mathrm C_{\mathrm A}^\bullet $
created on the spaces
\[
\mathrm C_{\mathrm A}^n = \Hom_\Lambda (\mathrm A_n, M)\simeq \Hom_{\Bbbk }
(\Bbbk V^{(n-1)}, M)
\]
with the differential $\Delta_{\mathrm A}^n: \mathrm C_{\mathrm A}^n \to \mathrm C_{\mathrm A}^{n+1}$ given by
\[
\Delta^n_{\mathrm A}(\psi) = \psi\dcobound_{n+1}
= \psi \mathrm f_n\dcobound_{n+1} \mathrm g_{n+1}
=\Delta_{\mathrm B}^n(\psi \mathrm f_n) g_{n+1}.
\]

The homotopy equivalence between $\mathrm A_\bullet $
and $\mathrm B_\bullet $
transforms $D_{\mathrm B}^n$
to 
\[
D_{\mathrm A}^n : \mathrm C_{\mathrm A}^n \to \mathrm C_{\mathrm A}^{n}
\]
such that 
\[
D_{\mathrm A}^n\psi = D_{\mathrm B}^n(\psi \mathrm f_n)\mathrm g_n.
\]

For every $\mathbf a\in \mathrm A_n$, 
we have 
\[
(D_{\mathrm A}^n\psi)(\mathbf a)
= \partial (\psi \mathrm f_n (\mathrm g_n (\mathbf a)))
-(\psi\mathrm f_n)(\partial_n(\mathrm g_n(\mathbf a)))
= \partial (\psi (\mathbf a))
-(\psi\mathrm f_n)(\partial_n(\mathrm g_n(\mathbf a))).
\]
\begin{proposition}[\hspace{-0,001cm}\cite{Ak}]
  For a conformal algebra $C$ and a conformal $C$-module $M$, 
  the reduced Hochschild complex 
  $\mathrm C^\bullet (C,M)$ is homotopy equivalent 
  to $\mathrm C_{\mathrm A}^\bullet /D_{\mathrm A}^\bullet \mathrm C_{\mathrm A}^\bullet $.
\end{proposition}

    \section{Application of Anick resolution on $\mathcal{A}_{+}(U(3))$}
    
    \begin{lemma} [\hspace{-0.001cm}\cite{Ak}]\label{lem:Reduction}
The elements of $\C^n$ are in one-to-one 
correspondence with 
scalar sequences $\alpha_{(i_1,i_2,\ldots,i_n)}$, 
$[i_1|i_2|\ldots |i_n]\in \mathrm A_n$.
\end{lemma}

\begin{theorem}\label{them:H^1}
For the conformal module $M_{(\alpha,\Delta)}$ 
over $U(3)$, where $\Delta\ne0$, we have 
\[
\dim_\Bbbk \Homol^1(U(3), M_{(\alpha,\Delta)})
= \begin{cases}
2, &  \Delta  = 1, \alpha= 0 \\
0, & \text{otherwise}. 
  \end{cases}
  \]
\end{theorem}

We are interested in the space
$\Homol^1(U(3),M) $ which is isomorphic 
to the space of non-coboundary cocycles in $\C^1 = \tilde \C^1/D^1\tilde \C^1$.
Suppose $\varphi \in \tilde \C^1 = \Hom_\Lambda (\A_1, M)$. 
By Lemma~\ref{lem:Reduction} we may assume 
$\varphi [n] = \alpha_n \in \Bbbk $
for $n\ge 0$.
Then the differential $\Delta^1\varphi$ takes the following 
values on the Anick 1-chains

\begin{align*}
(\Delta^1\varphi)[1|0] = & \ \varphi(\delta_2[1|0])=v(1)\alpha_0 -v(0)\alpha_1 - \alpha_0 = \Delta\alpha_0-(\partial+a)\alpha_1 - \alpha_0\\
=& \ (\Delta-1)\alpha_0-(\partial+a)\alpha_1,\\
(\Delta^1\varphi)[n|m]=& \ \varphi(\delta_2[n|m] = v(n)\alpha_m - \dfrac{nm}{n+m-1} v(1)\alpha_{n+m-1}\\
+& \ \dfrac{(n-1)(m-1)}{n+m-1} v(0)\alpha_{n+m} - \dfrac{n(n-1)}{n+m-1} \alpha_{n+m-1} \\
=& \  - \dfrac{nm}{n+m-1} \Delta \alpha_{n+m-1}
 + \dfrac{(n-1)(m-1)}{n+m-1} (\partial+a)\alpha_{n+m}\\
 -& \ \dfrac{n(n-1)}{n+m-1} \alpha_{n+m-1}.
\end{align*}

In order to find the constants that define $\Delta^1(\varphi +D^1\tilde \C^1) \in \C^2$
choose $\psi \in \tilde \C^2$
such that 
\[
\psi[1|0] =-\alpha_1, \quad \psi[n|m] = \dfrac{(n-1)(m-1)}{n+m-1} \alpha_{n+m},
\]
Then
\begin{align}\nonumber
(\Delta^1\varphi-D^2\psi)[n|m]=& \ - \dfrac{nm}{n+m-1} \Delta \alpha_{n+m-1} + \dfrac{(n-1)(m-1)}{n+m-1} a\alpha_{n+m}\\\nonumber
& \ - \dfrac{n(n-1)}{n+m-1} \alpha_{n+m-1}-\dfrac{n(n-2)(m-1)}{n+m-2} \alpha_{n+m-1}\\\nonumber
& \ -\dfrac{m(n-1)(m-2)}{n+m-2} \alpha_{n+m-1};\quad n\ge2, m\ge0,\\
(\Delta^1\varphi - D^2\psi )[1|0] =& \ (\Delta-1)\alpha_0-a\alpha_1,
\label{eq:2-cocycle}
\end{align}
for all $[n|m]\in \mathrm A_2$.
Therefore, $\varphi +D^1\tilde \C^1 $ is a 1-cocycle in $\C^1$
if and only if:

{\sc Case 1:} $\alpha=0$. Put $m=1$ in \eqref{eq:2-cocycle} to obtain 
\begin{align*}
(\Delta^1\varphi - D^2\psi )[1|0] =& \ (\Delta-1)\alpha_0=0,\\
(\Delta^1\varphi - D^2\psi )[n|1] =& \ -(\Delta +n-2)\alpha_n=0, \quad n\ge 2.
\end{align*}

For all other Anick 1-chains $[n|m]$, the desired value is proportional to $\alpha_{n+m-1}$, 
so $\alpha_1$  does not emerge in these expressions.

Therefore, if $\varphi -D^1\psi$ is a cocycle in $\C^1$ then $\alpha_n=0$ for all $n\ge 2$
except, maybe, for $n=2-\Delta$. The latter is impossible for $n\ge 3$ since 
\[
(\Delta^1 \varphi - D^2\psi)[n-1|2] = 
-\dfrac{2(n-1)}{n}\Delta \alpha_n
- \dfrac{(n-1)(n-2)}{n}\alpha_n - (n-3)\alpha_n = \dfrac{2}{n}\alpha_n.
\]

Finally, we obtain the description of cocycles in $\C^1$ for 
various $\Delta$:
\begin{itemize}
    \item $\Delta=1$: $\alpha_0$ and $\alpha_1$ take arbitrary values, $\alpha_n=0$ for $n\ge 2$; 
  \item $\Delta\ne 1$: $\alpha_1$ is arbitrary, $\alpha_0=\alpha_n=0$ for $n\ge 2$.
\end{itemize}

Coboundary cocycles in $\tilde \C^1$ are given by 
$\Delta^0 h$, where $h\in \Hom_\Lambda (\Lambda, M)$. 
Modulo $D^0\tilde C^0$, we may assume $h(1) = \beta u$, $\beta \in \Bbbk $.
Then 
$(\Delta^0h)[n] = v(n)\beta u$. 
Choose $\psi \in \tilde \C^1$ such that 
$\psi[0]=\beta  u$ and 
$\psi[n] = 0$ for $n\ge 1$. 
Then 
\[
(\Delta^0 h - D^1\psi )[n]
=\begin{cases}
 0, & n=0, \\
 (\Delta-1)\beta u, & n=1, \\
 0, & n\ge 2.
\end{cases}
\]
Hence, 
the space of coboundaries in $\C^1$
is 1-dimensional for $\Delta\ne 1$ and zero otherwise.

As a result, 
for the 1st Hochschild cohomology of $U(3)$ we have 
\[
\dim_\Bbbk \Homol^1(U(3), M(0,\Delta))
= \begin{cases}
2, & \Delta  = 1, \\
0, & \text{otherwise}. 
  \end{cases}
\]

{\sc Case 2:} $\alpha\ne0$. Put $m=0$ in \eqref{eq:2-cocycle} to obtain  
\begin{align*}
(\Delta^1\varphi - D^2\psi )[1|0] =& \ (\Delta-1)\alpha_0-a\alpha_1,\\
(\Delta^1\varphi - D^2\psi )[n|0] =& \ -a\alpha_n, \quad n\ge 2.
\end{align*}

If $\delta \varphi -D^1\psi$ is a cocycle in $\C^1$ then for all $n\ge 2$
\[
\begin{gathered}
  \alpha_n=0\quad \mbox{and} \quad
 \alpha_1=\dfrac{\Delta-1}{a}\alpha_0. 
\end{gathered}
\]

Coboundary cocycles in $\tilde \C^1$ are given by 
$\delta_0 h$, where $h\in \Hom_\Lambda (\Lambda, M)$. 
Modulo $D^0\tilde C^0$, we may assume $h(0) = cu$, $c=\frac{\alpha_0}{a} \in \Bbbk $.
Then 
$(\Delta^0h)[0] = v(0)c u=(\partial+a)cu$. 
Choose $\psi \in \tilde \C^1$ such that 
$\psi[0]=c  u$ and 
$\psi[n] = 0$ for $n\ge 1$. 
Then 
\[
(\Delta^0 h - D^1\psi )[n]
=\begin{cases}
 ac=\alpha_0, & n=0, \\
 0, & n\ge 1.
\end{cases}
\]
Hence, 
the space of coboundaries in $\C^1$
is 1-dimensional for $a\ne 0$.

As a result, 
for the 1st Hochschild cohomology of $U(3)$ we have 
\[
\dim_\Bbbk \Homol^1(U(3), M_{(\alpha,\Delta)})
= \begin{cases}
2, & \Delta  = 1, \alpha= 0 \\
0, & \text{otherwise}. 
  \end{cases}
\]

\begin{theorem}\label{thm:4.3} The
cohomology groups of $U(3)$ with the values in $M_{(\alpha,\Delta)}$ where $a\ne0$ and $\Delta \in \Bbbk$  are trivial.
\end{theorem}
\begin{proof}
By Lemma~\ref{lem:Reduction} we may assume 
$\varphi ([v(n)v(m)]) = \alpha_{(n,m)} , \varphi ([v(1)v(0)]) = \alpha_{(1,0)}\in \Bbbk $
for $n\ge2$ and $m\ge0$.
The differential $\Delta^2\varphi$ takes the following 
values on the Anick 2-chains:
    \[
    \begin{aligned}
    (\Delta^2\varphi)[v(n)v(m)v(p)]&{}=-\frac{nm}{n+m-1}\Delta \alpha_{(n+m-1,p)}\\
    &+\frac{(n-1)(m-1)}{n+m-1}(\partial+\alpha)\alpha_{(n+m,p)}-\frac{n(n-1)}{n+m-1}\alpha_{(n+m-1,p)}\\
    &+\frac{mp}{m+p-1}\Delta \alpha_{(n,m+p-1)}+\frac{(n-1)mp}{m+p-1}\alpha_{(n,m+p-1)}\\
    &-\frac{(m-1)(p-1)}{m+p-1}(\partial+\alpha)\alpha_{(n,m+p)}-\frac{n(m-1)(p-1)}{m+p-1}\alpha_{(n-1,m+p)}\\
    &+\frac{m(m-1)}{m+p-1}\alpha_{(n,m+p-1)};\quad n,m\ge2,p\ge0,\\
    (\Delta^2\varphi)[v(n)v(1)v(0)]&=-\Delta \alpha_{(n,0)}-(n-2)\alpha_{(n,0)}+(\partial+\alpha)\alpha_{(n,1)}+n\alpha_{(n-1,1)};\quad n\ge2.
    \end{aligned}
    \]
    \\
   Reduce the result by means of $D^3\psi $, where
$\psi \in \tilde C^3 $ is given by 
    \begin{align*}
    \psi(n,m,p)=& \ \frac{(n-1)(m-1)}{n+m-1}\alpha_{(n+m,p)}-\frac{(m-1)(p-1)}{m+p-1}\alpha_{(n,m+p)};\quad n,m\ge2,p\ge0,\\
   \psi(n,1,0)=& \ \alpha_{(n,1)};\quad n\ge2.
    \end{align*}

Then
    \begin{align*}
  (\Delta^2\varphi-D^3\psi)(n,m,p)=& \ -\frac{nm}{n+m-1}\Delta \alpha_{(n+m-1,p)}\\
    +& \ \frac{(n-1)(m-1)}{n+m-1}\alpha\alpha_{(n+m,p)}-\frac{n(n-1)}{n+m-1}\alpha_{(n+m-1,p)}\\
    +& \ \frac{mp}{m+p-1}\Delta \alpha_{(n,m+p-1)}+\frac{(n-1)mp}{m+p-1}\alpha_{(n,m+p-1)}\\
    -& \ \frac{(m-1)(p-1)}{m+p-1}\alpha\alpha_{(n,m+p)}+\frac{m(m-1)}{m+p-1}\alpha_{(n,m+p-1)}\\
    - & \ \frac{n(n-2)(m-1)}{n+m-2}\alpha_{(n+m-1,p)}-\frac{p(n-1)(m-1)}{n+m-1}\alpha_{(n+m,p-1)}\\
    \end{align*}
    \begin{align*}
  -& \ \frac{m(n-1)(m-2)}{n+m-2}\alpha_{(n+m-1,p)}+\frac{m(m-2)(p-1)}{m+p-2}\alpha_{(n,m+p-1)}\\
  +& \ \frac{p(m-1)(p-2)}{m+p-2}\alpha_{(n,m+p-1)};\quad n,m\ge2,p\ge0,\\ 
(\Delta^2\varphi-D^3\psi)(n,1,0)= & \ -(\Delta+n-2) \alpha_{(n,0)}+\alpha\alpha_{(n,1)};\quad n\ge2. 
  \end{align*}
   Hence, $\varphi +D^2\tilde \C^2$ is a 2-cocycle in $\C^2$
if and only if and for $p=0$
  \begin{align*}
  -\alpha\alpha_{(n,m)}=& \ -\frac{nm}{n+m-1}\Delta \alpha_{(n+m-1,0)}+\frac{(n-1)(m-1)}{n+m-1}\alpha\alpha_{(n+m,0)}-\frac{n(n-1)}{n+m-1}\alpha_{(n+m-1,0)}\\
  -& \ \frac{n(n-2)(m-1)}{n+m-2}\alpha_{(n+m-1,0)}-\frac{m(n-1)(m-2)}{n+m-2}\alpha_{(n+m-1,0)};\quad n,m\ge2,\\
   \alpha\alpha_{(n,1)}=& \ (\Delta+n-2) \alpha_{(n,0)};\quad n\ge2.
    \end{align*}

Therefore, cocycles in $\C^2$ are determined by $\alpha_{(n,0)}$
for $n\ge 2$.

     Choose $\varphi_1\in \tilde C^1, \psi_1\in \tilde C^2$, such that
     \[
\begin{gathered}
\varphi_1[n]=\beta_{n};\quad n\ge0,\quad \psi_1[1|0]=\beta_1,\quad
\psi_1[n|m]= \dfrac{(n-1)(m-1)}{n+m-1} \beta_{n+m};\quad n\ge2,m\ge0
    \end{gathered}
    \]
    where
    \[
    \beta_n=
    \begin{cases}
   \dfrac{-\alpha_{(n,0)}}{\alpha};& n\ge1.\\
   0; & \text{otherwise}.
    \end{cases}
    \]
Then $\Delta^{1}\varphi_1-D^{2}\psi_1$
represents a coboundary in $\C^{2}$, and
    \[
 (\Delta^{1}\varphi_1 - D^2\psi_1 )[n|0] = -\alpha\beta_n=\alpha_{(n,0)}=\varphi(n,0);\quad n\ge1.   
    \]
 Hence, every 2-cocycle is a coboundary.
 
 Similar way, evaluate $\Delta^{n}\varphi - D^{n+1}\psi$:
 \begin{align*}
 (\Delta^{n}\varphi - D^{n+1}\psi )[i_1|i_2|\ldots|i_{n+1}]&=\sum_{j=1}^n\frac{(-1)^ji_ji_{j+1}}{i_j+i_{j+1}-1}\Delta\alpha_{(i_1,\ldots,i_j+i_{j+1}-1,\ldots,i_{n+1})}\\
 & +\sum^n_{j=1}(-1)^{j}\frac{(i_j-1)(i_{j+1}-1)}{i_j+i_{j+1}-1}\alpha\alpha_{(i_1,\ldots,i_j+i_{j+1},\ldots,i_{n+1})}\\
  & +\sum^n_{j=1}(-1)^{j}\frac{i_j(i_j-1)}{i_j+i_{j+1}-1}\alpha_{(i_1,\ldots,i_j+i_{j+1}-1,\ldots,i_{n+1})}\\
  &+\sum^n_{j=2}\sum^{j-1}_{t=1}(-1)^{j}\frac{i_ji_{j+1}}{i_j+i_{j+1}-1}(i_t-1)\alpha_{(i_1,\ldots,i_j+i_{j+1}-1,\ldots,i_{n+1})}\\
  \end{align*}

  \begin{align*}
   &+\sum^{n+1}_{t=3}\sum^{t-2}_{j=1}(-1)^{j}i_t\frac{(i_j-1)(i_{j+1}-1)}{i_j+i_{j+1}-1}\alpha_{(i_1,\ldots ,i_j+i_{j+1},\ldots, i_t-1,\ldots ,i_{n+1})}\\
     &+\sum^n_{t=1}(-1)^{t}i_t\frac{(i_t-2)(i_{t+1}-1)}{i_t+i_{t+1}-2}\alpha_{(i_1,\ldots ,i_t+i_{t+1}-1,\ldots ,i_{n+1})}\\
     &+\sum^{n}_{t=1}(-1)^{t}i_{t+1}\frac{(i_t-1)(i_{t+1}-2)}{i_t+i_{t+1}-2}\alpha_{(i_1,\ldots ,i_t+i_{t+1}-1,\ldots ,i_{n+1})},
     \end{align*}
     
 \begin{align*}
 (\Delta^{n}\varphi - D^{n+1}\psi )[i_1|i_2|\ldots|i_{n-1}|1|0]&=\sum_{j=1}^{n-2}\frac{(-1)^ji_ji_{j+1}}{i_j+i_{j+1}-1}\Delta\alpha_{(i_1,\ldots,i_j+i_{j+1}-1,\ldots,i_{n-1},1,0)}\\
& +(-1)^{n-1}\Delta\alpha_{(i_1,i_2,\ldots,i_{n-1},0)}\\
 & +\sum^{n-2}_{j=1}(-1)^{j}\frac{(i_j-1)(i_{j+1}-1)}{i_j+i_{j+1}-1}\alpha\alpha_{(i_1,\ldots,i_j+i_{j+1},\ldots,i_{n-1},1,0)}\\
 & +(-1)^{n-1}a\alpha_{(i_1,i_2,\ldots,i_{n-1},1)}\\
  & +\sum^{n-2}_{j=1}(-1)^{j}\frac{i_j(i_j-1)}{i_j+i_{j+1}-1}\alpha_{(i_1,\ldots,i_j+i_{j+1}-1,\ldots,i_{n-1},1,0)}\\
  & +(-1)^{n-1}\alpha_{(i_1,i_2,\ldots,i_{n-1},0)}\\
  &+\sum^{n-2}_{j=1}\sum^{j-1}_{t=1}(-1)^{j}\frac{i_ji_{j+1}}{i_j+i_{j+1}-1}(i_t-1)\\
  &\times\alpha_{(i_1,\ldots,i_j+i_{j+1}-1,\ldots,i_{n-1},1,0)}\\
  &+\sum^{n-1}_{t=1}(-1)^{j}(i_t-1)\alpha_{(i_1,i_2,\ldots,i_{n-1},0)}\\
   &+\sum^{n-1}_{t=3}\sum^{t-2}_{j=1}(-1)^{j}i_t\frac{(i_j-1)(i_{j+1}-1)}{i_j+i_{j+1}-1}\\
   &\times\alpha_{(i_1,\ldots ,i_j+i_{j+1},\ldots, i_t-1,\ldots ,i_{n-1},1,0)}\\
&+\sum^{n-2}_{t=1}(-1)^{t}i_t\frac{(i_t-2)(i_{t+1}-1)}{i_t+i_{t+1}-2}\alpha_{(i_1,\ldots ,i_t+i_{t+1}-1,\ldots ,i_{n-1},1,0)}\\
  &+\sum^{n-2}_{t=1}(-1)^{t}i_{t+1}\frac{(i_t-1)(i_{t+1}-2)}{i_t+i_{t+1}-2}\alpha_{(i_1,\ldots ,i_t+i_{t+1}-1,\ldots ,i_{n-1},1,0)},
 \end{align*}
where $[i_1|i_2|\ldots|i_n|i_{n+1}]\in V^{n}$. We have that $\varphi - D^{n}\tilde \C^{n} $ is a $n$-cocycle in $\C^n$
if and only if 
 \begin{align*}
(-1)^n\alpha\alpha_{(i_1,i_2,\ldots,i_{n})}&=\sum_{j=1}^{n-1}\frac{(-1)^ji_ji_{j+1}}{i_j+i_{j+1}-1}\Delta\alpha_{(i_1,\ldots,i_j+i_{j+1}-1,\ldots,i_n,0)}\\
 & +\sum^{n-1}_{j=1}(-1)^{j}\frac{(i_j-1)(i_{j+1}-1)}{i_j+i_{j+1}-1}\alpha\alpha_{(i_1,\ldots,i_j+i_{j+1},\ldots,i_n,0)}\\
  & +\sum^n_{j=1}(-1)^{j}\frac{i_j(i_j-1)}{i_j+i_{j+1}-1}\alpha_{(i_1,\ldots,i_j+i_{j+1}-1,\ldots,i_n,0)}\\
  &+\sum^{n-1}_{j=2}\sum^{j-1}_{t=1}(-1)^{j}\frac{i_ji_{j+1}}{i_j+i_{j+1}-1}(i_t-1)\alpha_{(i_1,\ldots,i_j+i_{j+1}-1,\ldots,i_n,0)}\\
   &+\sum^{n}_{t=3}\sum^{t-2}_{j=1}(-1)^{j}i_t\frac{(i_j-1)(i_{j+1}-1)}{i_j+i_{j+1}-1}\alpha_{(i_1,\ldots ,i_j+i_{j+1},\ldots, i_t-1,\ldots ,i_n,0)}\\
     &+\sum^{n-1}_{t=1}(-1)^{t}i_t\frac{(i_t-2)(i_{t+1}-1)}{i_t+i_{t+1}-2}\alpha_{(i_1,\ldots ,i_t+i_{t+1}-1,\ldots ,i_{n+1})}\\
     &+\sum^{n-1}_{t=1}(-1)^{t}i_{t+1}\frac{(i_t-1)(i_{t+1}-2)}{i_t+i_{t+1}-2}\alpha_{(i_1,\ldots ,i_t+i_{t+1}-1,\ldots ,i_n,0)},
     \end{align*}
 \begin{align*}
 (-1)^n\alpha\alpha_{(i_1,i_2,\ldots,i_{n-1},1)}&=\sum_{j=1}^{n-2}\frac{(-1)^ji_ji_{j+1}}{i_j+i_{j+1}-1}\Delta\alpha_{(i_1,\ldots,i_j+i_{j+1}-1,\ldots,i_{n-1},1,0)}\\
& +(-1)^{n-1}\Delta\alpha_{(i_1,i_2,\ldots,i_{n-1},0)}\\
 & +\sum^{n-2}_{j=1}(-1)^{j}\frac{(i_j-1)(i_{j+1}-1)}{i_j+i_{j+1}-1}\alpha\alpha_{(i_1,\ldots,i_j+i_{j+1},\ldots,i_{n-1},1,0)}\\
  & +\sum^{n-2}_{j=1}(-1)^{j}\frac{i_j(i_j-1)}{i_j+i_{j+1}-1}\alpha_{(i_1,\ldots,i_j+i_{j+1}-1,\ldots,i_{n-1},1,0)}\\
  & +(-1)^{n-1}\alpha_{(i_1,i_2,\ldots,i_{n-1},0)}\\
  &+\sum^{n-2}_{j=1}\sum^{j-1}_{t=1}(-1)^{j}\frac{i_ji_{j+1}}{i_j+i_{j+1}-1}(i_t-1)\alpha_{(i_1,\ldots,i_j+i_{j+1}-1,\ldots,i_{n-1},1,0)}\\
  &+\sum^{n-1}_{t=1}(-1)^{j}(i_t-1)\alpha_{(i_1,i_2,\ldots,i_{n-1},0)}\\
  \end{align*}
  \begin{align*}
   &+\sum^{n-1}_{t=3}\sum^{t-2}_{j=1}(-1)^{j}i_t\frac{(i_j-1)(i_{j+1}-1)}{i_j+i_{j+1}-1}\alpha_{(i_1,\ldots ,i_j+i_{j+1},\ldots, i_t-1,\ldots ,i_{n-1},1,0)}\\
        &+\sum^{n-2}_{t=1}(-1)^{t}i_t\frac{(i_t-2)(i_{t+1}-1)}{i_t+i_{t+1}-2}\alpha_{(i_1,\ldots ,i_t+i_{t+1}-1,\ldots ,i_{n-1},1,0)}\\
     &+\sum^{n-2}_{t=1}(-1)^{t}i_{t+1}\frac{(i_t-1)(i_{t+1}-2)}{i_t+i_{t+1}-2}\alpha_{(i_1,\ldots ,i_t+i_{t+1}-1,\ldots ,i_{n-1},1,0)}.
 \end{align*}
 
 Therefore, cocycles in $\C^n$ are determined by $\alpha_{(i_1,i_2,\ldots , ,i_{n-1},0)}$ where  $i_1,\ldots,i_{n-2}\ge2$, $i_{n-1}\ge1$. 
 
     Choose $\varphi_1\in \tilde \C^{n-1}, \psi_1\in \tilde \C^n$, such that
     \begin{align*}
     \varphi_1(i_1,i_2,\ldots,i_{n-1})=& \ \beta_{(i_1,i_2,\ldots,i_{n-1})}\\
     \psi_1(i_1,i_2,\ldots,i_{n})=& \ \sum^{n-1}_{j=1}(-1)^{j}\frac{(i_j-1)(i_{j+1}-1)}{i_j+i_{j+1}-1}\beta_{(i_1,\ldots,i_j+i_{j+1},\ldots,i_n)}\\
     \psi_1(i_1,i_2,\ldots,i_{n-2},1,0){}=& \ \sum^{n-2}_{j=1}(-1)^{j}\frac{(i_j-1)(i_{j+1}-1)}{i_j+i_{j+1}-1}\beta_{(i_1,\ldots,i_j+i_{j+1},\ldots,i_{n-2},1,0)}\\
     +& \ (-1)^{n-1}\beta_{(i_1,i_2,\ldots,i_{n-1},1)},
     \end{align*}
where
    \[
    \beta_{(i_1,i_2,\ldots,i_{n-1})}=
    \begin{cases}
   \dfrac{(-1)^{n-1}\alpha_{(i_1,i_2,\ldots,i_{n-1},0)}}{\alpha};& i_1,\ldots,i_{n-2}\ge2, i_{n-1}\ge1\\
   0; & \text{otherwise}.
    \end{cases}
    \]
    
Then $\Delta^{n-1}\varphi_1-D^{n}\psi_1$
represents a coboundary in $C^{n}$, and for all $ i_1,\ldots,i_{n-2}\ge2 $ and $ i_{n-1}\ge1$, we get
    \begin{align*}
 (\Delta^{n-1}\varphi_1 - D^n\psi_1 )(i_1,i_2,\ldots,i_{n-1},0)& = (-1)^{n-1}\alpha\alpha_{(i_1,i_2,\ldots,i_{n-1})}=\alpha_{(i_1,i_2,\ldots,i_{n-1},0)}\\
 &=\varphi_{(i_1,i_2,\ldots,i_{n-1},0)}.   
    \end{align*}
    
 Hence, the element 
$\varphi -D^{n}\tilde \C^{n}$
is a coboundary in $\C^{n}$.
 \end{proof}
 
 \begin{theorem}\label{them:H^2}
 For the conformal module $M_{(0,\Delta)}$ where $\Delta\ne0$ over $U(3)$, we have 
 \[\dim_\Bbbk\Homol^2(U(3),M_{(0,\Delta)})=
    \begin{cases}
     1, &\Delta=1 \\
    0, & \Delta\ne0,1.
    \end{cases}
    \]
    \end{theorem}  
    
    \begin{proof}
  For $\alpha=0$, the differential $\Delta^{2}\varphi$ takes the following 
values on the Anick 2-chains:
   
   \[
    \begin{aligned}
    (\Delta^{2}\varphi)[v(n)v(m)v(p)]&{}=-\frac{nm}{n+m-1}\Delta \alpha_{(n+m-1,p)}\\
    &+\frac{(n-1)(m-1)}{n+m-1}\partial\alpha_{(n+m,p)}-\frac{n(n-1)}{n+m-1}\alpha_{(n+m-1,p)}\\
    &+\frac{mp}{m+p-1}\Delta \alpha_{(n,m+p-1)}+\frac{(n-1)mp}{m+p-1}\alpha_{(n,m+p-1)}\\
    &-\frac{(m-1)(p-1)}{m+p-1}\partial\alpha_{(n,m+p)}-\frac{n(m-1)(p-1)}{m+p-1}\alpha_{(n-1,m+p)}\\
    &+\frac{m(m-1)}{m+p-1}\alpha_{(n,m+p-1)};\quad n,m\ge2, p\ge0,\\
    (\Delta^{2}\varphi)[v(n)v(1)v(0)]&{}=-\Delta \alpha_{(n,0)}-(n-2)\alpha_{(n,0)}+\partial\alpha_{(n,1)}+n\alpha_{(n-1,1)};\quad n\ge2.
    \end{aligned}
    \]
   Reduce the result by means of $D^3\psi $, where
$\psi \in \tilde C^3 $ is given by 
    \begin{align*}
    \psi(n,m,p)=& \ \frac{(n-1)(m-1)}{n+m-1}\alpha_{(n+m,p)}-\frac{(m-1)(p-1)}{m+p-1}\alpha_{(n,m+p)};\quad n,m\ge2, p\ge0,\\
   \psi(n,1,0)=& \ \alpha_{(n,1)};\quad n\ge2.
    \end{align*}
 Namely, 
    \begin{align*}
  (\Delta^2\varphi-D^3\psi)(n,m,p)=& \ -\frac{nm}{n+m-1}\Delta \alpha_{(n+m-1,p)}-\frac{n(n-1)}{n+m-1}\alpha_{(n+m-1,p)}\\
 +& \ \frac{mp}{m+p-1}\Delta \alpha_{(n,m+p-1)}+\frac{(n-1)mp}{m+p-1}\alpha_{(n,m+p-1)}\\
 +& \ \frac{m(m-1)}{m+p-1}\alpha_{(n,m+p-1)}
 -\frac{n(n-2)(m-1)}{n+m-2}\alpha_{(n+m-1,p)}\\
 -& \ \frac{m(n-1)(m-2)}{n+m-2}\alpha_{(n+m-1,p)}+\frac{m(m-2)(p-1)}{m+p-2}\alpha_{(n,m+p-1)}\\
  -& \ \frac{p(n-1)(m-1)}{n+m-1}\alpha_{(n+m,p-1)}+\frac{p(m-1)(p-2)}{m+p-2}\alpha_{(n,m+p-1)};\\
  & \ n,m\ge2, p\ge0, (m,p)\ne(2,0),\\
(\Delta^2\varphi-D^3\psi)(n,2,0)=& \ -\frac{2n}{n+1}\Delta \alpha_{(n+1,0)}-\frac{n(n-1)}{n+1}\alpha_{(n+1,0)}-(n-2)\alpha_{(n+1,0)},\\
(\Delta^2\varphi-D^3\psi)(n,1,0)=& \ -(\Delta+n-2) \alpha_{(n,0)};\quad n\ge2.
    \end{align*}
    
  Suppose $\varphi - D^{3}\tilde \C^{2} $
is a cocycle in $\C^{2}$. Since
\[
(\Delta^2 \varphi -D^3 \psi)(n,1,0)=0,\quad (\Delta^2\varphi-D^3\psi)(n-1,2,0)=0,\quad (\Delta^2\varphi-D^3\psi)(n,m,1)=0,
\]
therefore
\[
\alpha_{(2,0)}=0,\quad \alpha_{(n,0)}=0;\quad n\ge3,
\]
    \begin{align}\nonumber
  -(\Delta+n+m-3)\alpha_{(n,m)}=& \ -\frac{nm}{n+m-1}\Delta \alpha_{(n+m-1,1)}\\\nonumber
  -& \ \frac{n(n-1)}{n+m-1}\alpha_{(n+m-1,1)}-\frac{n(n-2)(m-1)}{n+m-2}\alpha_{(n+m-1,1)}\\
  -& \ \frac{m(n-1)(m-2)}{n+m-2}\alpha_{(n+m-1,1)};\quad n,m\ge2.
 \label{2-cocycle}
   \end{align}

Choose $
\varphi_1\in \tilde C^1, \psi_1 \in \tilde C^2
$
such that 
\[
\varphi_1[n] =\beta_n,\quad \psi_1[1|0] =\beta_1, \quad \psi_1[n|m] = \dfrac{(n-1)(m-1)}{n+m-1} \beta_{n+m},
\]
and 
\begin{align*}
(\Delta^1\varphi_1-D^2\psi_1)[n|m]=& \ - \dfrac{nm}{n+m-1} \Delta \beta_{n+m-1}\\
 -& \ \dfrac{n(n-1)}{n+m-1} \beta_{n+m-1}-\dfrac{n(n-2)(m-1)}{n+m-2} \beta_{n+m-1}\\
 -& \ \dfrac{m(n-1)(m-2)}{n+m-2} \beta_{n+m-1};\quad n\ge2, m\ge0,\\
(\Delta^1\varphi_1 - D^2\psi_1 )[1|0] =& \ (\Delta-1)\beta_0,
\end{align*}
where
\[\beta_{k}=
    \begin{cases}
    \dfrac{\alpha_{(1,0)}}{\Delta-1},& k=0, \Delta\ne1,\\ \\
    -\dfrac{\alpha_{(2,1)}}{\Delta},& k=2, \Delta\ne0,\\ \\
     -\dfrac{\alpha_{(k,1)}}{\Delta+k-2}, & k\ge3,k\ne 2-\Delta, \\ \\
    -\dfrac{k}{2}\alpha_{(k-1,2)}, & k\ge3,k=2-\Delta.
    \end{cases}
    \]
 Then $\Delta^1\varphi_1 $ is a coboundary in $\tilde \C^2$, and
 \begin{align*}
     (\Delta^1\varphi_1 - D^2\psi_1 )[2|1] =& \ -\Delta \beta_2=\alpha_{(2,1)}=\varphi(2,1);\quad \Delta\ne0,\\
(\Delta^1\varphi_1 - D^2\psi_1 )[1|0] =& \ (\Delta-1)\beta_0=\alpha_{(1,0)}=\varphi(1,0);\quad \Delta\ne1,\\
(\Delta^1\varphi_1 - D^2\psi_1 )[k|1] =& \ (-\Delta -k+2)\beta_k=\alpha_{(k,1)}=\varphi(k,1);\quad k\ne 2-\Delta,\\
(\Delta^1\varphi_1-D^2\psi_1)[k-1|2]=& \ - \dfrac{2(k-1)}{k} \Delta \beta_{k} - \dfrac{(k-1)(k-2)}{k} \beta_{k}-(k-3) \beta_{k}\\
=& \ -\dfrac{2}{k}\beta_{k}=\alpha_{(k-1,2)}=\varphi(k-1,2);\quad k=2-\Delta.
\end{align*}

Finally, we obtain the description of cocycles in $\C^2$ for 
various $\Delta$:
\begin{itemize}
 \item There are $n_1, m_1$ such that $\Delta=3-n_1-m_1=2-k$, where  $n_1,m_1\ge2$. From \eqref{2-cocycle} we have
\[
0=\frac{n_1m_1(m_1+1)+m_1(m_1-1)(m_1-2)}{(n_1+m_1-1)(n_1+m_1-2)}\alpha_{(n_1+m_1-1,1)}.
\]
Therefore
\[
\alpha_{(k,1)}=0;\quad k=2-\Delta.
\]
If $\varphi -D^3\psi$ is a cocycle in $\C^2$ then  $(\Delta^2 \varphi -D^3\psi)(n_1,m_1-1,2)=0$; $m_1\ge3$, if and only if:
    \begin{align*}
  -\frac{2}{m_1}\alpha_{(n_1,m_1)}=& \ -\frac{n_1(m_1-1)}{n_1+m_1-2}\Delta \alpha_{(n_1+m_1-2,2)}-\frac{n_1(n_1-1)}{n_1+m_1-2}\alpha_{(n_1+m_1-2,2)}\\
 -& \ \frac{n_1(n_1-2)(m_1-2)}{n_1+m_1-3}\alpha_{(n_1+m_1-2,2)}\\
 -& \ \frac{(m_1-1)(n_1-1)(m_1-3)}{n_1+m_1-3}\alpha_{(n_1+m_1-2,2)}.
   \end{align*}
Hence, in this case, cocycles in $\C^2$ are either zero or determined by  $\alpha_{(k-1,2)}$  such that $\Delta+n_1+m_1-3=\Delta+k-2=0$.

\item   $\Delta=n+m-3\ne0$, in this case, cocycles in $\C^2$ are either zero or determined by  $\alpha_{(k,1)}$ (see \eqref{2-cocycle}) such that $\Delta=2-k\ne0$. 
    \end{itemize}
    
    In both cases, every 2-cocycle is a coboundary or zero except $\alpha_{(1,0)}$ when $\Delta=1$.
  \end{proof}
  
\begin{theorem}\label{them:H^n}
For $n\ge3$, the
cohomology groups of $U(3)$ with the values in $M_{(0,\Delta)}$ where $0\ne\Delta \in \Bbbk$  are trivial.
    \end{theorem}
\begin{proof}
 For $n=3$, in the same way as the previous theory, we may assume 
 \[\varphi ([v(n)v(m)v(p)]) = \alpha_{(n,m,p)} , \varphi ([v(n)v(1)v(0)]) = \alpha_{(n,1,0)}\in \Bbbk,
 \]
for $n,m\ge2$ and $p\ge0$.

 The differential $\Delta^3\varphi$ takes the following 
values on the Anick 3-chains:
    \begin{align*}
      (\Delta^3\varphi)[v(n)v(m)v(p)v(q)]&{}=-\frac{nm}{n+m-1}\Delta\alpha_{(n+m-1,p,q)}-\frac{p(p-1)}{p+q-1}\alpha_{(n,m,p+q-1)}\\
      &+\frac{(n-1)(m-1)}{n+m-1}\partial\alpha_{(n+m,p,q)}-\frac{n(n-1)}{n+m-1}\alpha_{(n+m-1,p,q)} \\
      &+\frac{mp}{m+p-1}\Delta\alpha_{(n,m+p-1,q)}+\frac{(n-1)mp}{m+p-1}\alpha_{(n,m+p-1,q)}\\
      &-\frac{(m-1)(p-1)}{m+p-1}\partial\alpha_{(n,m+p,q)}+\frac{(p-1)(q-1)}{p+q-1}\partial\alpha_{(n,m,p+q)}\\
      &-\frac{pq}{p+q-1}\Delta\alpha_{(n,m,p+q-1)}-\frac{(m-1)pq}{p+q-1}\alpha_{(n,m,p+q-1)}\\
      \end{align*}
      \begin{align*}
      &-\frac{(n-1)pq}{p+q-1}\alpha_{(n,m,p+q-1)}+\frac{m(m-1)}{m+p-1}\alpha_{(n,m+p-1,q)}\\
      &+\frac{m(p-1)(q-1)}{p+q-1}\alpha_{(n,m-1,p+q)};\quad n,m,p\ge2, q\ge0,
    \end{align*}
     \[
    \begin{aligned}
     (\Delta^3\varphi)[v(n)v(m)v(1)v(0)]&{}=-\frac{nm}{n+m-1}\Delta\alpha_{(n+m-1,1,0)}\\ 
      &+\frac{(n-1)(m-1)}{n+m-1}\partial\alpha_{(n+m,1,0)}-\frac{n(n-1)}{n+m-1}\alpha_{(n+m-1,1,0)} \\
      &+\Delta\alpha_{(n,m,0)}+(n+m-3)\alpha_{(n,m,0)}-m\alpha_{(n,m-1,1)}\\
      &-n\alpha_{(n-1,m,1)}-\partial\alpha_{(n,m,1)};\quad n,m\ge2.
    \end{aligned}
    \]
    Reduce the result by means of $D^4\psi $, where
$\psi \in \tilde C^4 $ is given by 
\[
    \begin{aligned}
      \psi[v(n)v(m)v(p)v(q)]=
      & \ \frac{(n-1)(m-1)}{n+m-1}\alpha_{(n+m,p,q)}\\
      -& \ \frac{(m-1)(p-1)}{m+p-1}\alpha_{(n,m+p,q)}\\
      +& \ \frac{(p-1)(q-1)}{p+q-1}\alpha_{(n,m,p+q)};\quad n,m,p\ge2,q\ge 0,
    \end{aligned}
    \]
    \[
    \begin{aligned}
      \psi[v(n)v(m)v(1)v(0)]{}=
      \frac{(n-1)(m-1)}{n+m-1}\alpha_{(n+m,1,0)}-\alpha_{(n,m,1)};\quad n,m\ge2.
    \end{aligned}
    \]
    
   Namely,
    \begin{align*}
      (\Delta^3\varphi-D^4\psi)[v(n)v(m)v(p)v(q)]&{}=-\frac{nm}{n+m-1}\Delta\alpha_{(n+m-1,p,q)}-\frac{n(n-1)}{n+m-1}\alpha_{(n+m-1,p,q)} \\
      &+\frac{mp}{m+p-1}\Delta\alpha_{(n,m+p-1,q)}+\frac{(n-1)mp}{m+p-1}\alpha_{(n,m+p-1,q)}\\
      &-\frac{pq}{p+q-1}\Delta\alpha_{(n,m,p+q-1)}-\frac{(m-1)pq}{p+q-1}\alpha_{(n,m,p+q-1)}\\
      &-\frac{(n-1)pq}{p+q-1}\alpha_{(n,m,p+q-1)}+\frac{m(m-1)}{m+p-1}\alpha_{(n,m+p-1,q)}\\
      &-\frac{p(p-1)}{p+q-1}\alpha_{(n,m,p+q-1)}-\frac{n(n-2)(m-1)}{n+m-2}\alpha_{(n+m-1,p,q)}\\
      &-\frac{m(n-1)(m-2)}{n+m-2}\alpha_{(n+m-1,p,q)}\\
      &+\frac{m(m-2)(p-1)}{m+p-2}\alpha_{(n,m+p-1,q)}\\
      \end{align*}
      \begin{align*}
      &-\frac{p(n-1)(m-1)}{n+m-1}\alpha_{(n+m,p-1,q)}+\frac{p(m-1)(p-2)}{m+p-2}\alpha_{(n,m+p-1,q)}\\
      &-\frac{p(p-2)(q-1)}{p+q-2}\alpha_{(n,m,p+q-1)}-\frac{q(n-1)(m-1)}{n+m-1}\alpha_{(n+m,p,q-1)}\\
      &+\frac{q(m-1)(p-1)}{m+p-1}\alpha_{(n,m+p,q-1)}-\frac{q(p-1)(q-2)}{p+q-2}\alpha_{(n,m,p+q-1)},
    \end{align*}
    \begin{align*}
     (\Delta^3\varphi-D^4\psi)[v(n)v(m)v(1)v(0)]&{}=-\frac{nm}{n+m-1}\Delta\alpha_{(n+m-1,1,0)}-\frac{n(n-1)}{n+m-1}\alpha_{(n+m-1,1,0)} \\
      &+\Delta\alpha_{(n,m,0)}+(n+m-3)\alpha_{(n,m,0)}\\
       &-\frac{n(n-2)(m-1)}{n+m-2}\alpha_{(n+m-1,1,0)}\\
       &-\frac{m(n-1)(m-2)}{n+m-2}\alpha_{(n+m-1,1,0)}.
    \end{align*}
    
    Hence, $\varphi -D^4\tilde \C^4$ is a 3-cocycle in $\C^4$
for various $\Delta$.

{\sc Case 1:} $\Delta+n+m-3\ne0$; $n,m\ge2.$
    \begin{align*}\label{eq:45.1}
     -(\Delta+n+m-3)\alpha_{(n,m,0)}&{}=-\frac{nm}{n+m-1}\Delta\alpha_{(n+m-1,1,0)}-\frac{n(n-1)}{n+m-1}\alpha_{(n+m-1,1,0)} \\
       &-\frac{n(n-2)(m-1)}{n+m-2}\alpha_{(n+m-1,1,0)}-\frac{m(n-1)(m-2)}{n+m-2}\alpha_{(n+m-1,1,0)}.
    \end{align*}
    
    {\sc Case 2:} $\Delta+n+m+p-4\ne0$; $n,m,p\ge2.$
     \begin{align*}
      (\Delta+n+m+p-4)\alpha_{(n,m,p)}&=-\frac{nm}{n+m-1}\Delta\alpha_{(n+m-1,p,1)}-\frac{n(n-1)}{n+m-1}\alpha_{(n+m-1,p,1)}\\
      &+\frac{mp}{m+p-1}\Delta\alpha_{(n,m+p-1,1)}+\frac{(n-1)mp}{m+p-1}\alpha_{(n,m+p-1,1)}\\
      &+\frac{m(m-1)}{m+p-1}\alpha_{(n,m+p-1,1)}-\frac{n(n-2)(m-1)}{n+m-2}\alpha_{(n+m-1,p,1)}\\
      &-\frac{m(n-1)(m-2)}{n+m-2}\alpha_{(n+m-1,p,1)}\\
      &+\frac{m(m-2)(p-1)}{m+p-2}\alpha_{(n,m+p-1,1)}\\
      &-\frac{p(n-1)(m-1)}{n+m-1}\alpha_{(n+m,p-1,1)}\\
      &+\frac{p(m-1)(p-2)}{m+p-2}\alpha_{(n,m+p-1,1)}\\
      &-\frac{(n-1)(m-1)}{n+m-1}\alpha_{(n+m,p,0)}+\frac{(m-1)(p-1)}{m+p-1}\alpha_{(n,m+p,0)}.
    \end{align*}
    
    {\sc Case 3:} $\Delta+n_1+m_1-3=0$; $n_1,m_1\ge2.$
      \begin{align*}
     \alpha_{(n_1+m_1-1,1,0)}=& \ 0,\\
      \frac{-2m_1+4}{m_1}\alpha_{(n_1,m_1,0)}=& \ -\frac{n_1(m_1-1)}{n_1+m_1-2}\Delta\alpha_{(n_1+m_1-2,2,0)}-\frac{n_1(n_1-1)}{n_1+m_1-2}\alpha_{(n_1+m_1-2,2,0)} \\
      -& \ \frac{n_1(n_1-2)(m_1-2)}{n_1+m_1-3}\alpha_{(n_1+m_1-2,2,0)}\\
      -& \ \frac{(m_1-1)(n_1-1)(m_1-3)}{n_1+m_1-3}\alpha_{(n_1+m_1-2,2,0)}.
    \end{align*}
    
    {\sc Case 4:} $\Delta+n_2+m_2+p_2-4=0$; $n_2,m_2,p_2\ge2.$
    \begin{align*}
    \dfrac{2p_2-4}{p_2}\alpha_{(n_2,m_2,p_2)}&=-\frac{n_2m_2}{n_2+m_2-1}\Delta\alpha_{(n_2+m_2-1,p_2-1,2)}-\frac{n_2(n_2-1)}{n_2+m_2-1}\alpha_{(n_2+m_2-1,p_2-1,2)} \\
      &+\frac{m_2(p_2-1)}{m_2+p_2-2}\Delta\alpha_{(n_2,m_2+p_2-2,2)}+\frac{(n_2-1)m_2(p_2-1)}{m_2+p_2-2}\alpha_{(n_2,m_2+p_2-2,2)}\\
      &+\frac{m_2(m_2-1)}{m_2+p_2-2}\alpha_{(n_2,m_2+p_2-2,2)}-\frac{n_2(n_2-2)(m_2-1)}{n_2+m_2-2}\alpha_{(n_2+m_2-1,p_2-1,2)}\\
      &-\frac{m_2(n_2-1)(m_2-2)}{n_2+m_2-2}\alpha_{(n_2+m_2-1,p_2-1,2)}\\
      &+\frac{m_2(m_2-2)(p_2-2)}{m_2+p_2-3}\alpha_{(n_2,m_2+p_2-2,2)}\\
      &-\frac{(p_2-1)(n_2-1)(m_2-1)}{n_2+m_2-1}\alpha_{(n_2+m_2,p_2-2,2)}\\
      &+\frac{(p_2-1)(m_2-1)(p_2-3)}{m_2+p_2-3}\alpha_{(n_2,m_2+p_2-2,2)}\\
      &-\frac{2(n_2-1)(m_2-1)}{n_2+m_2-1}\alpha_{(n_2+m_2,p_2-1,1)}+\frac{2(m_2-1)(p_2-2)}{m_2+p_2-2}\alpha_{(n_2,m_2+p_2-1,1)}.
    \end{align*}
    
 The Case 4 can be obtained from Case 3 when $p_2=2$ and $m_1\ge3$ where 
 \[
 \Delta+n_1+m_1-3=\Delta+n_1+m_1-1+2-4=\Delta+n_2+m_2+p_2-4=0.
 \]
 
 Therefore
 \begin{align*}
  \alpha_{(n_2,m_2,2)}=& \ \alpha_{(n_1,m_1-1,2)},\\
      \frac{2}{m_2}\alpha_{(n_2,m_2,1)}=&  -\frac{n_2(m_2-1)}{n_2+m_2-2}\Delta\alpha_{(n_2+m_2-2,2,1)}-\frac{n_2(n_2-1)}{n_2+m_2-2}\alpha_{(n_2+m_2-2,2,1)}\\
      &-\frac{n_2(n_2-2)(m_2-2)}{n_2+m_2-3}\alpha_{(n_2+m_2-2,2,1)}\\
      &-\frac{(m_2-1)(n_2-1)(m_2-3)}{n_2+m_2-3}\alpha_{(n_2+m_2-2,2,1)}\\
      &-\frac{(n_2-1)(m_2-2)}{n_2+m_2-2}\alpha_{(n_2+m_2-1,2,0)}+\frac{m_2-1}{m_2+1}\alpha_{(n_2,m_2+1,0)}.
    \end{align*}

 Choose $\varphi_1\in \tilde C^2$ and $\psi_1\in \tilde C^3$ such that 
    \[
    \varphi_1(n,m)=\beta_{(n,m)}\quad,\quad \psi_1(n,1,0)=\beta_{(n,1)};\quad n\ge2, m\ge0.
    \]
    \[
     \psi_1(n,m,p)=\frac{(n-1)(m-1)}{n+m-1}\beta_{(n+m,p)}-\frac{(m-1)(p-1)}{m+p-1}\beta_{(n,m+p)};\quad n,m\ge2, p\ge0.
    \]
    Namely,
    \begin{align*}
  (\Delta^2\varphi_1-D^3\psi_1)(n,m,p)=& -\frac{nm}{n+m-1}\Delta \beta_{(n+m-1,p)}-\frac{n(n-1)}{n+m-1}\beta_{(n+m-1,p)}\\
    &+\frac{mp}{m+p-1}\Delta \beta_{(n,m+p-1)}+\frac{(n-1)mp}{m+p-1}\beta_{(n,m+p-1)}\\
    &+\frac{m(m-1)}{m+p-1}\beta_{(n,m+p-1)}
 -\frac{n(n-2)(m-1)}{n+m-2}\beta_{(n+m-1,p)}\\
 &-\frac{m(n-1)(m-2)}{n+m-2}\beta_{(n+m-1,p)}+\frac{m(m-2)(p-1)}{m+p-2}\beta_{(n,m+p-1)}\\
  &-\frac{p(n-1)(m-1)}{n+m-1}\beta_{(n+m,p-1)}+\frac{p(m-1)(p-2)}{m+p-2}\beta_{(n,m+p-1)}, 
   \end{align*}
    
    \begin{align*}
  (\Delta^2\varphi_1-D^3\psi_1)(n,m,1)=&-\frac{nm}{n+m-1}\Delta \beta_{(n+m-1,1)}-\frac{n(n-1)}{n+m-1}\beta_{(n+m-1,1)}\\
&+(\Delta+n+m-3) \beta_{(n,m)}-\frac{n(n-2)(m-1)}{n+m-2}\beta_{(n+m-1,1)}\\
  &-\frac{m(n-1)(m-2)}{n+m-2}\beta_{(n+m-1,1)}-\frac{(n-1)(m-1)}{n+m-1}\beta_{(n+m,0)},
   \end{align*}

    \begin{align*}
  (\Delta^2\varphi_1-D^3\psi_1)(n,m,0)=&-\frac{nm}{n+m-1}\Delta\beta_{(n+m-1,0)}-\frac{n(n-1)}{n+m-1}\beta_{(n+m-1,0)} \\
       &-\frac{n(n-2)(m-1)}{n+m-2}\beta_{(n+m-1,0)}-\frac{m(n-1)(m-2)}{n+m-2}\beta_{(n+m-1,0)},
   \end{align*}
    
    \[
    (\Delta^2\varphi_1-D^3\psi_1)(n,1,0)=-(\Delta+n-2)\beta_{(n,0)},
    \] 
     where
     \[\beta_{(k,0)}=
    \begin{cases}
    -\frac{\alpha_{(k,1,0)}}{\Delta+k-2}; & k\ne 2-\Delta \\
   0; & \text{otherwise}.
    \end{cases}
    \]

     Then $\Delta^2\varphi_1 $ is a coboundary in $\tilde \C^3$, and
      \[
    (\Delta^3\varphi_1-D^4\psi_1)(k,1,0)=\alpha_{(k,1,0)}=\varphi_{(k,1,0)};\quad k\ne 2-\Delta.
    \]
    
    Let us repeat the construction
with new values of $\beta$'s to get $\varphi^{(i)}_1$ and $\psi^{(i)}_1$ for $i=1,2,3,$ as following:

    \begin{itemize}
\item For 
\[
\beta_{(k,0)}=
    \begin{cases}
    \frac{k}{2}\alpha_{(k,2,0)}; & k= 2-\Delta,\\
    0; & \text{otherwise},
    \end{cases}
    \]
    we have
    \begin{align*}
    (\Delta^3\varphi^{(1)}_1-D^4\psi^{(1)}_1)(k,2,0)=& \ \alpha_{(k,2,0)}=\varphi_{(k,2,0)};\quad k= 2-\Delta,\\
    (\Delta^3\varphi^{(1)}_1-D^4\psi^{(1)}_1)(k,1,0)=& \ 0;\quad k\ne 2-\Delta.
    \end{align*}
 \item  For
    \[\beta_{(k,1)}=
    \begin{cases}
    \frac{k}{2}\alpha_{(k,2,1)}; & k=2-\Delta,\\
    0; & \text{otherwise},
    \end{cases}
    \]
we have
      \begin{align*}
    (\Delta^2\varphi^{(2)}_1-D^3\psi^{(2)}_1)(k,2,1)=& \ \alpha_{(k,2,1)}=\varphi_{(k,2,1)},\\
    (\Delta^3\varphi^{(2)}_1-D^4\psi^{(2)}_1)(k,1,0)=& \ 0,\\
    (\Delta^3\varphi^{(2)}_1-D^4\psi^{(2)}_1)(k,2,0)=& \ 0.
    \end{align*}
\item     Also if we put
    \[\beta_{(n,m)}=
    \begin{cases}
    \frac{\alpha_{(n,m,1)}}{\Delta+n+m-3}; & \Delta+n+m-3\ne0, n,m\ge2 \\
    -\frac{m}{2}\alpha_{(n,m-1,2)}; & \Delta+n+m-3=0, n,m\ge2.\\
    0; & \text{otherwise},
    \end{cases}
    \]
    we get  
$\varphi^{(3)}_1$ and $\psi^{(3)}_1$ such that  
      \begin{align*}
    (\Delta^2\varphi^{(3)}_1-D^3\psi^{(3)}_1)(n,m,1)=& \ \varphi_{(n,m,1)},\\
    (\Delta^2\varphi^{(3)}_1-D^3\psi^{(3)}_1)(n,m-1,2)= & \ \varphi_{(n,m-1,2)},\\
    (\Delta^2\varphi^{(3)}_1-D^3\psi^{(3)}_1)(n,m,p)=& \ 0;\quad \text{otherwise}.
    \end{align*}
    Hence, every 3-cocycle is a coboundary.
    \end{itemize}
    
 The differential $\Delta^4\varphi$ takes the following 
values on the Anick 4-chains:
\[
    \begin{aligned}
      (\Delta^4\varphi)[v(n)v(m)v(p)v(q)v(r)]&{}=-\frac{nm}{n+m-1}\Delta\alpha_{(n+m-1,p,q,r)}+\frac{mp}{m+p-1}\Delta\alpha_{(n,m+p-1,q,r)}\\
      &-\frac{pq}{p+q-1}\Delta\alpha_{(n,m,p+q-1,r)}+\frac{qr}{q+r-1}\Delta\alpha_{(n,m,p,q+r-1)}\\
        & +\frac{(n-1)(m-1)}{n+m-1}\partial\alpha_{(n+m,p,q,r)}-\frac{(m-1)(p-1)}{m+p-1}\partial\alpha_{(n,m+p,q,r)}\\
         &+\frac{(p-1)(q-1)}{p+q-1}\partial\alpha_{(n,m,p+q,r)}-\frac{(q-1)(r-1)}{q+r-1}\partial\alpha_{(n,m,p,q+r)}\\
      &-\frac{n(n-1)}{n+m-1}\alpha_{(n+m-1,p,q,r)} +\frac{m(m-1)}{m+p-1}\alpha_{(n,m+p-1,q,r)}\\
      \end{aligned}
      \]
      \[
      \begin{aligned}
      &-\frac{p(p-1)}{p+q-1}\alpha_{(n,m,p+q-1,r)}+\frac{q(q-1)}{q+r-1}\alpha_{(n,m,p,q+r-1)}\\
      &+\frac{(n-1)mp}{m+p-1}\alpha_{(n,m+p-1,q,r)}-\frac{(n-1)pq}{p+q-1}\alpha_{(n,m,p+q-1,r)}\\
      &-\frac{(m-1)pq}{p+q-1}\alpha_{(n,m,p+q-1,r)}+\frac{(n-1)qr}{q+r-1}\alpha_{(n,m,p,q+r-1)}\\
      &+\frac{(m-1)qr}{q+r-1}\alpha_{(n,m,p,q+r-1)}+\frac{(p-1)qr}{q+r-1}\alpha_{(n,m,p,q+r-1)}\\
     & -\frac{n(m-1)(p-1)}{m+p-1}\alpha_{(n-1,m+p,q,r)}+\frac{n(p-1)(q-1)}{p+q-1}\alpha_{(n-1,m,p+q,r)}\\
      &+\frac{m(p-1)(q-1)}{p+q-1}\alpha_{(n,m-1,p+q,r)}-\frac{n(q-1)(r-1)}{q+r-1}\alpha_{(n-1,m,p,q+r)}\\
      &-\frac{m(q-1)(r-1)}{q+r-1}\alpha_{(n,m-1,p,q+r)}-\frac{p(q-1)(r-1)}{q+r-1}\alpha_{(n,m,p-1,q+r)},
    \end{aligned}
    \]
     \[
    \begin{aligned}
     (\Delta^4\varphi)[v(n)v(m)v(p)v(1)v(0)]&{}=-\frac{nm}{n+m-1}\Delta\alpha_{(n+m-1,p,1,0)}+\frac{mp}{m+p-1}\Delta\alpha_{(n,m+p-1,1,0)}\\ 
      &+\frac{(n-1)(m-1)}{n+m-1}\partial\alpha_{(n+m,p,1,0)}-\frac{(m-1)(p-1)}{m+p-1}\partial\alpha_{(n,m+p,1,0)}\\
      &+\partial\alpha_{(n,m,p,1)}-\frac{n(n-1)}{n+m-1}\alpha_{(n+m-1,p,1,0)}\\
      &+\frac{m(m-1)}{m+p-1}\alpha_{(n,m+p-1,1,0)}+\frac{n(m-1)(p-1)}{m+p-1}\alpha_{(n-1,m+p,1,0)}\\
      &+\frac{(n-1)mp}{m+p-1}\alpha_{(n,m+p,1,0)}+\Delta\alpha_{(n,m,0)}+(n+m-3)\alpha_{(n,m,0)}\\
      &-n\alpha_{(n-1,m,p,1)}+m\alpha_{(n,m-1,p,1)}-p\alpha_{(n,m,p-1,1)}.
    \end{aligned}
    \]
    
   Reduce the result by means of $D^5\psi $, where
$\psi \in \tilde C^5 $ is given by 
\[
    \begin{aligned}
      \psi[v(n)v(m)v(p)v(q)v(r)]&{}= +\frac{(n-1)(m-1)}{n+m-1}\alpha_{(n+m,p,q,r)}-\frac{(m-1)(p-1)}{m+p-1}\alpha_{(n,m+p,q,r)}\\
         &+\frac{(p-1)(q-1)}{p+q-1}\alpha_{(n,m,p+q,r)}-\frac{(q-1)(r-1)}{q+r-1}\alpha_{(n,m,p,q+r)}.
    \end{aligned}
    \]
    \[
    \begin{aligned}
      \psi[v(n)v(m)v(p)v(1)v(0)]&{}= +\frac{(n-1)(m-1)}{n+m-1}\alpha_{(n+m,p,1,0)}-\frac{(m-1)(p-1)}{m+p-1}\alpha_{(n,m+p,1,0)}\\
         &+\alpha_{(n,m,p,1)}.
    \end{aligned}
    \]
    
    Namely,
        \begin{align*}
      (\Delta^4\varphi-D^5\psi)[v(n)v(m)v(p)v(q)v(r)]=&-\frac{nm}{n+m-1}\Delta\alpha_{(n+m-1,p,q,r)}\\
      &+\frac{mp}{m+p-1}\Delta\alpha_{(n,m+p-1,q,r)}\\
      \end{align*}
      \begin{align*}
     &-\frac{pq}{p+q-1}\Delta\alpha_{(n,m,p+q-1,r)}+\frac{qr}{q+r-1}\Delta\alpha_{(n,m,p,q+r-1)}-\frac{n(n-1)}{n+m-1}\alpha_{(n+m-1,p,q,r)}\\
     &+\frac{m(m-1)}{m+p-1}\alpha_{(n,m+p-1,q,r)}-\frac{p(p-1)}{p+q-1}\alpha_{(n,m,p+q-1,r)}+\frac{q(q-1)}{q+r-1}\alpha_{(n,m,p,q+r-1)}\\ 
      &+\frac{(n-1)mp}{m+p-1}\alpha_{(n,m+p-1,q,r)}-\frac{(n-1)pq}{p+q-1}\alpha_{(n,m,p+q-1,r)}-\frac{(m-1)pq}{p+q-1}\alpha_{(n,m,p+q-1,q,r)}\\
      &+\frac{(n-1)qr}{q+r-1}\alpha_{(n,m,p,q+r-1)}+\frac{(m-1)qr}{q+r-1}\alpha_{(n,m,p,q+r-1)}+\frac{(p-1)qr}{q+r-1}\alpha_{(n,m,p,q+r-1)}\\
      &-\frac{n(n-2)(m-1)}{n+m-2}\alpha_{(n+m-1,p,q,r)}-\frac{m(n-1)(m-2)}{n+m-2}\alpha_{(n+m-1,p,q,r)}\\
      &+\frac{m(m-2)(p-1)}{m+p-2}\alpha_{(n,m+p-1,q,r)}-\frac{p(n-1)(m-1)}{n+m-1}\alpha_{(n+m,p-1,q,r)} \\
      &+\frac{p(m-1)(p-2)}{m+p-2}\alpha_{(n,m+p-1,q,r)}-\frac{p(p-2)(q-1)}{p+q-2}\alpha_{(n,m,p+q-1,r)} \\
       &-\frac{q(n-1)(m-1)}{n+m-1}\alpha_{(n+m,p,q-1,r)}+\frac{q(m-1)(p-1)}{m+p-1}\alpha_{(n,m+p,q-1,r)} \\
      &-\frac{q(p-1)(q-2)}{p+q-2}\alpha_{(n,m,p+q-1,r)}+\frac{q(q-2)(r-1)}{q+r-2}\alpha_{(n,m,p,q+r-1)}\\
      &-\frac{r(n-1)(m-1)}{n+m-1}\alpha_{(n+m,p,q,r-1)}+\frac{r(m-1)(p-1)}{m+p-1}\alpha_{(n,m+p,q,r-1)} \\
      &-\frac{r(p-1)(q-1)}{p+q-1}\alpha_{(n,m,p+q,r-1)}+\frac{r(q-1)(r-2)}{q+r-2}\alpha_{(n,m,p,q+r-1)},
    \end{align*}
    \begin{align*}
     (\Delta^4\varphi-D^5\psi)[v(n)v(m)v(p)v(1)v(0)]=& -\frac{nm}{n+m-1}\Delta\alpha_{(n+m-1,p,1,0)}\\
     &+\frac{mp}{m+p-1}\Delta\alpha_{(n,m+p-1,1,0)}\\ 
     &-\frac{n(n-1)}{n+m-1}\alpha_{(n+m-1,p,1,0)}\\
     &+\frac{m(m-1)}{m+p-1}\alpha_{(n,m+p-1,1,0)} \\
      &-(\Delta+n+m+p-4)\alpha_{(n,m,p,0)}\\
      &-\frac{n(n-2)(m-1)}{n+m-2}\alpha_{(n+m-1,p,1,0)}\\
      &+\frac{n(m-1)(p-1)}{m+p-1}\alpha_{(n-1,m+p,1,0)}\\
      &-\frac{m(n-1)(m-2)}{n+m-2}\alpha_{(n+m-1,p,1,0)}\\
      \end{align*}
      \begin{align*}
      &+\frac{n(m-2)(p-1)}{m+p-2}\alpha_{(n,m+p-1,1,0)}-\frac{p(n-1)(m-1)}{n+m-1}\alpha_{(n+m,p-1,1,0)}\\
      &+\frac{p(m-1)(p-2)}{m+p-2}\alpha_{(n,m+p-1,1,0)}.
    \end{align*}
    
     Hence, $\varphi -D^5\tilde \C^5$ is a 4-cocycle in $\C^4$.
     
     {\sc Case 1:} $\Delta+n+m+p+q-5\ne0$; $n,m,p,q\ge2.$
    \[
    \begin{aligned}
      -(\Delta+n+m+p+q-5)\alpha_{(n,m,p,q)}&=-\frac{nm}{n+m-1}\Delta\alpha_{(n+m-1,p,q,1)}\\
      &+\frac{mp}{m+p-1}\Delta\alpha_{(n,m+p-1,q,1)}-\frac{pq}{p+q-1}\Delta\alpha_{(n,m,p+q-1,1)}\\
      &-\frac{n(n-1)}{n+m-1}\alpha_{(n+m-1,p,q,1)}+\frac{m(m-1)}{m+p-1}\alpha_{(n,m+p-1,q,1)}\\
      &-\frac{p(p-1)}{p+q-1}\alpha_{(n,m,p+q-1,1)}+\frac{(n-1)mp}{m+p-1}\alpha_{(n,m+p-1,q,1)}\\
      &-\frac{(n-1)pq}{p+q-1}\alpha_{(n,m,p+q-1,q,1)}-\frac{(m-1)pq}{p+q-1}\alpha_{(n,m,p+q-1,q,1)}\\
      &-\frac{n(n-2)(m-1)}{n+m-2}\alpha_{(n+m-1,p,q,1)}\\
      &-\frac{m(n-1)(m-2)}{n+m-2}\alpha_{(n+m-1,p,q,1)}\\
      &+\frac{m(m-2)(p-1)}{m+p-2}\alpha_{(n,m+p-1,q,1)}\\
      &-\frac{p(n-1)(m-1)}{n+m-1}\alpha_{(n+m,p-1,q,1)}\\
      &+\frac{p(m-1)(p-2)}{m+p-2}\alpha_{(n,m+p-1,q,1)}\\
      &-\frac{p(p-2)(q-1)}{p+q-2}\alpha_{(n,m,p+q-1,1)}\\
      &-\frac{q(n-1)(m-1)}{n+m-1}\alpha_{(n+m,p,q-1,1)}\\
      &+\frac{q(m-1)(p-1)}{m+p-1}\alpha_{(n,m+p,q-1,1)}\\
      &-\frac{q(p-1)(q-2)}{p+q-2}\alpha_{(n,m,p+q-1,1)}\\
      &-\frac{(n-1)(m-1)}{n+m-1}\alpha_{(n+m,p,q,0)}\\
      &+\frac{(m-1)(p-1)}{m+p-1}\alpha_{(n,m+p,q,0)}\\
      &-\frac{(p-1)(q-1)}{p+q-1}\alpha_{(n,m,p+q,0)}.
    \end{aligned}
    \]
    
     {\sc Case 2:} $\Delta+n+m+p-4\ne0$; $n,m,p\ge2.$
       \[
    \begin{aligned}
     (\Delta+n+m+p-4)\alpha_{(n,m,p,0)}&{}=-\frac{nm}{n+m-1}\Delta\alpha_{(n+m-1,p,1,0)}+\frac{mp}{m+p-1}\Delta\alpha_{(n,m+p-1,1,0)}\\ 
      &-\frac{n(n-1)}{n+m-1}\alpha_{(n+m-1,p,1,0)}+\frac{m(m-1)}{m+p-1}\alpha_{(n,m+p-1,1,0)} \\
      &-\frac{n(n-2)(m-1)}{n+m-2}\alpha_{(n+m-1,p,1,0)}+\frac{(n-1)mp}{m+p-1}\alpha_{(n,m+p-1,1,0)}\\
      &-\frac{m(n-1)(m-2)}{n+m-2}\alpha_{(n+m-1,p,1,0)}\\
      &+\frac{m(m-2)(p-1)}{m+p-2}\alpha_{(n,m+p-1,1,0)}\\
      &-\frac{p(n-1)(m-1)}{n+m-1}\alpha_{(n+m,p-1,1,0)}\\
      &+\frac{p(m-1)(p-2)}{m+p-2}\alpha_{(n,m+p-1,1,0)}.
    \end{aligned}
    \]
         
      {\sc Case 3:} $\Delta+n_2+m_2+p_2-4=0$;  $n_2,m_2,p_2\ge2.$
      \[
    \begin{aligned}
     \frac{2p_2-4}{p_2}\alpha_{(n_2,m_2,p_2,0)}&{}=-\frac{n_2m_2}{n_2+m_2-1}\Delta\alpha_{(n_2+m_2-1,p_2-1,2,0)}+\frac{m_2(p_2-1)}{m_2+p_2-2}\Delta\alpha_{(n_2,m_2+p_2-2,2,0)}\\
      &-\frac{n_2(n_2-1)}{n_2+m_2-1}\alpha_{(n_2+m_2-1,p_2-1,2,0)}+\frac{m_2(m_2-1)}{m_2+p_2-2}\alpha_{(n_2,m_2+p_2-2,2,0)} \\
      &+\frac{m_2(n_2-1)(p_2-1)}{m_2+p_2-2}\alpha_{(n_2,m_2+p_2-2,2,0)}\\
      &-\frac{n_2(n_2-2)(m_2-1)}{n_2+m_2-2}\alpha_{(n_2+m_2-1,p_2-1,2,0)}\\
      &-\frac{m_2(n_2-1)(m_2-2)}{n_2+m_2-2}\alpha_{(n_2+m_2-1,p_2-1,2,0)}\\
      &+\frac{m_2(m_2-2)(p_2-2)}{m_2+p_2-3}\alpha_{(n_2,m_2+p_2-2,2,0)}\\
      &-\frac{(p_2-1)(n_2-1)(m_2-1)}{n_2+m_2-1}\alpha_{(n_2+m_2,p_2-3,2,0)}\\
      &+\frac{(p_2-1)(m_2-1)(p_2-3)}{m_2+p_2-3}\alpha_{(n_2,m_2+p_2-2,2,0)}\\
      &-\frac{2(n_2-1)(m_2-1)}{n_2+m_2-1}\alpha_{(n_2+m_2,p_2-1,1,0)}\\
      &+\frac{2(m_2-1)(p_2-2)}{m_2+p_2-3}\alpha_{(n_2,m_2+p_2-1,1,0)}.
    \end{aligned}
    \]
    
 We can obtain from Case 4 a new special case when $p_2=2$. Let $n_1=n_2$ and let $m_1=m_2+1\ge3$ then 
 \[
 \Delta+n_1+m_1-3=\Delta+n_1+m_1-1+2-4=\Delta+n_2+m_2+p_2-4=0.
 \]
 
 Therefore
\[
    \begin{aligned}
      \alpha_{(n_2,m_2,2,0)}&=\alpha_{(n_1,m_1-1,2,0)}\\
     \frac{2}{m_2}\alpha_{(n_2,m_2,1,0)}&=-\frac{n_2(m_2-1)}{n_2+m_2-2}\Delta\alpha_{(n_2+m_2-2,2,1,0)}
-\frac{n_2(n_2-1)}{n_2+m_2-2}\alpha_{(n_2+m_2-2,2,1,0)} \\
      &-\frac{n_2(n_2-2)(m_2-2)}{n_2+m_2-3}\alpha_{(n_2+m_2-2,2,1,0)}\\
      &-\frac{(m_2-1)(n_2-2)(m_2-3)}{n_2+m_2-3}\alpha_{(n_2+m_2-2,2,1,0)},
      \end{aligned}
\]
   \[
    \begin{aligned}
     \frac{-2}{p_2}\alpha_{(n_2,m_2,p_2,1)}&=-\frac{n_2m_2}{n_2+m_2-1}\Delta\alpha_{(n_2+m_2-1,p_2-1,2,1)}+\frac{m_2(p_2-1)}{m_2+p_2-2}\Delta\alpha_{(n_2,m_2+p_2-2,2,1)}\\
     & -\frac{n_2(n_2-1)}{n_2+m_2-1}\alpha_{(n_2+m_2-1,p_2-1,2,1)}+\frac{m_2(m_2-1)}{m_2+p_2-2}\alpha_{(n_2,m_2+p_2-2,2,1)}\\
    &+\frac{(n_2-1)m_2(p_2-1)}{m_2+p_2-2}\alpha_{(n_2,m_2+p_2-2,2,1)}\\
    &-\frac{n_2(n_2-2)(m_2-1)}{n_2+m_2-2}\alpha_{(n_2+m_2-1,p_2-1,2,1)}\\
      &-\frac{m_2(n_2-1)(m_2-2)}{n_2+m_2-2}\alpha_{(n_2+m_2-1,p_2-1,2,1)}\\
      &+\frac{m_2(m_2-2)(p_2-2)}{m_2+p_2-3}\alpha_{(n_2,m_2+p_2-2,2,1)}\\
      &-\frac{(p_2-1)(n_2-1)(m_2-1)}{n_2+m_2-1}\alpha_{(n_2+m_2,p_2-2,2,1)}\\
      &+\frac{(p_2-1)(m_2-1)(p_2-3)}{m_2+p_2-3}\alpha_{(n_2,m_2+p_2-2,2,1)}\\
      &-\frac{(n_2-1)(m_2-1)}{n_2+m_2-1}\alpha_{(n_2+m_2,p_2-1,2,0)}\\
      &+\frac{(m_2-1)(p_2-2)}{m_2+p_2-2}\alpha_{(n_2,m_2+p_2-1,2,0)} \\
      &-\frac{(p_2-2)}{p_2}\alpha_{(n_2,m_2,p_2+1,0)}.
    \end{aligned}
    \]

{\sc Case 4:} $\Delta+n_3+m_3+p_3+q_3-5=0$;  $n_3,m_3,p_3\ge2,q_3\ge3.$
\[
    \begin{aligned}
     \frac{-2q_3+4}{q_3}\alpha_{(n_3,m_3,p_3,q_3)}&{}=-\frac{n_3m_3}{n_3+m_3-1}\Delta\alpha_{(n_3+m_3-1,p_3,q_3-1,2)}\\
     &+\frac{m_3p_3}{m_3+p_3-1}\Delta\alpha_{(n_3,m_3+p_3-1,q_3-1,2)}-\frac{p_3(q_3-1)}{p_3+q_3-2}\Delta\alpha_{(n_3,m_3,p_3+q_3-2,2)}\\
     &-\frac{n_3(n_3-1)}{n_3+m_3-1}\alpha_{(n_3+m_3-1,p_3,q_3-1,2)}+\frac{m_3(m_3-1)}{m_3+p_3-1}\alpha_{(n_3,m_3+p_3-1,q_3-1,2)} \\
     \end{aligned}
     \]
     \[
     \begin{aligned}
     &-\frac{p_3(p_3-1)}{p_3+q_3-2}\alpha_{(n_3,m_3,p_3+q_3-2,2)}+\frac{(n_3-1)m_3p_3}{m_3+p_3-1}\alpha_{(n_3,m_3+p_3-1,q_3-1,2)}\\
     &-\frac{(n_3-1)p_3(q_3-1)}{p_3+q_3-2}\alpha_{(n_3,m_3,p_3+q_3-2,2)}-\frac{(m_3-1)p_3(q_3-1)}{p_3+q_3-2}\alpha_{(n_3,m_3,p_3+q_3-2,2)}\\
     &-\frac{n_3(n_3-2)(m_3-1)}{n_3+m_3-2}\alpha_{(n_3+m_3-1,p_3,q_3-1,2)}-\frac{m_3(n_3-1)(m_3-2)}{n_3+m_3-2}\alpha_{(n_3+m_3-1,p_3,q_3-1,2)}\\
     &+\frac{m_3(m_3-2)(p_3-1)}{m_3+p_3-2}\alpha_{(n_3,m_3+p_3-1,q_3-1,2)}-\frac{p_3(n_3-1)(m_3-1)}{n_3+m_3-1}\alpha_{(n_3+m_3,p_3-1,q_3-1,2)}\\ 
     &+\frac{p_3(m_3-1)(p_3-2)}{m_3+p_3-2}\alpha_{(n_3,m_3+p_3-1,q_3-1,2)}-\frac{p_3(p_3-2)(q_3-2)}{p_3+q_3-3}\alpha_{(n_3,m_3,p_3+q_3-2,2)}\\
     &-\frac{(q_3-1)(n_3-1)(m_3-1)}{n_3+m_3-1}\alpha_{(n_3+m_3,p_3,q_3-2,2)}+\frac{(q_3-1)(m_3-1)(p_3-1)}{m_3+p_3-1}\alpha_{(n_3,m_3+p_3,q_3-2,2)}\\
     &-\frac{(q_3-1)(p_3-1)(q_3-3)}{m_3+p_3-3}\alpha_{(n_3,m_3,p_3+q_3-2,2)}-\frac{2(n_3-1)(m_3-1)}{n_3+m_3-1}\alpha_{(n_3+m_3,p_3,q_3-1,1)}\\
     &+\frac{2(m_3-1)(p_3-1)}{m_3+p_3-1}\alpha_{(n_3,m_3+p_3,q_3-1,1)}-\frac{2(p_3-1)(q_3-2)}{p_3+q_3-2}\alpha_{(n_3,m_3,p_3+q_3-1,1)}.
    \end{aligned}
    \]          
    
 Choose $\varphi_1\in \tilde \C^3$ and $\psi_1\in \tilde \C^4$ such that 
     \begin{align*}
     \varphi_1[v(n)v(m)v(p)]&=\beta_{(n,m,p)};\quad n,m\ge2, p\ge0,\\
     \varphi_1[v(n)v(1)v(0)]&=\beta_{(n,1,0)};\quad n\ge2,\\
      \psi_1[v(n)v(m)v(p)v(q)]&=
      \frac{(n-1)(m-1)}{n+m-1}\beta_{(n+m,p,q)}\\
      &-\frac{(m-1)(p-1)}{m+p-1}\beta_{(n,m+p,q)}\\
      &+\frac{(p-1)(q-1)}{p+q-1}\beta_{(n,m,p+q)};\quad n,m,p\ge2,q\ge 0,\\
      \psi_1[v(n)v(m)v(1)v(0)]&=
      \frac{(n-1)(m-1)}{n+m-1}\beta_{(n+m,1,0)}-\beta_{(n,m,1)};\quad n,m\ge2.
    \end{align*}

Namely,
    \begin{align*}
      (\Delta^3\varphi_1-D^4\psi_1)[v(n)v(m)v(p)v(q)]&=-\frac{nm}{n+m-1}\Delta\beta_{(n+m-1,p,q)}-\frac{n(n-1)}{n+m-1}\beta_{(n+m-1,p,q)} \\
      &+\frac{mp}{m+p-1}\Delta\beta_{(n,m+p-1,q)}+\frac{(n-1)mp}{m+p-1}\beta_{(n,m+p-1,q)}\\
      &-\frac{pq}{p+q-1}\Delta\beta_{(n,m,p+q-1)}-\frac{(m-1)pq}{p+q-1}\beta_{(n,m,p+q-1)}\\
      &-\frac{(n-1)pq}{p+q-1}\beta_{(n,m,p+q-1)}+\frac{m(m-1)}{m+p-1}\beta_{(n,m+p-1,q)}\\
      \end{align*}

      \begin{align*}
      &-\frac{p(p-1)}{p+q-1}\beta_{(n,m,p+q-1)}-\frac{n(n-2)(m-1)}{n+m-2}\beta_{(n+m-1,p,q)}\\
      &-\frac{m(n-1)(m-2)}{n+m-2}\beta_{(n+m-1,p,q)}+\frac{m(m-2)(p-1)}{m+p-2}\beta_{(n,m+p-1,q)}\\
      &-\frac{p(n-1)(m-1)}{n+m-1}\beta_{(n+m,p-1,q)}+\frac{p(m-1)(p-2)}{m+p-2}\beta_{(n,m+p-1,q)}\\
      &-\frac{p(p-2)(q-1)}{p+q-2}\beta_{(n,m,p+q-1)}-\frac{q(n-1)(m-1)}{n+m-1}\beta_{(n+m,p,q-1)}\\
      &+\frac{q(m-1)(p-1)}{m+p-1}\beta_{(n,m+p,q-1)}-\frac{q(p-1)(q-2)}{p+q-2}\beta_{(n,m,p+q-1)},
    \end{align*}
    \begin{align*}
     (\Delta^3\varphi_1-D^4\psi_1)[v(n)v(m)v(1)v(0)]&=-\frac{nm}{n+m-1}\Delta\beta_{(n+m-1,1,0)}-\frac{n(n-1)}{n+m-1}\beta_{(n+m-1,1,0)} \\
       &-\frac{n(n-2)(m-1)}{n+m-2}\beta_{(n+m-1,1,0)}\\
       &-\frac{m(n-1)(m-2)}{n+m-2}\beta_{(n+m-1,1,0)}\\
        &+(\Delta+n+m-3)\beta_{(n,m,0)},
    \end{align*}
    \begin{align*}
      (\Delta^3\varphi_1-D^4\psi_1)[v(n)v(m)v(p)v(0)]&=-\frac{nm}{n+m-1}\Delta\beta_{(n+m-1,p,0)}-\frac{n(n-1)}{n+m-1}\beta_{(n+m-1,p,0)}\\
      &+\frac{mp}{m+p-1}\Delta\beta_{(n,m+p-1,0)}+\frac{(n-1)mp}{m+p-1}\beta_{(n,m+p-1,0)}\\
      &+\frac{m(m-1)}{m+p-1}\beta_{(n,m+p-1,0)}-\frac{n(n-2)(m-1)}{n+m-2}\beta_{(n+m-1,p,0)}\\
      &-\frac{m(n-1)(m-2)}{n+m-2}\beta_{(n+m-1,p,0)}\\
      &+\frac{m(m-2)(p-1)}{m+p-2}\beta_{(n,m+p-1,0)}\\
      &-\frac{p(n-1)(m-1)}{n+m-1}\beta_{(n+m,p-1,0)}\\
      &+\frac{p(m-1)(p-2)}{m+p-2}\beta_{(n,m+p-1,0)},
    \end{align*}
    
    where

    \begin{itemize}
        \item for \[\beta_{(n,m,0)}=
    \begin{cases}
    \frac{\alpha_{(n,m,1,0)}}{\Delta+n+m-3};& \Delta+n+m-3\ne0,n,m\ge2,\\
    0; & \text{otherwise},
    \end{cases}
    \]
 then $\Delta^3\varphi_1 $ is a coboundary in $\tilde \C^4$, and
   \[
   (\Delta^3\varphi_1-D^4\psi_1)[v(n)v(m)v(1)v(0)]=\varphi(n,m,1,0);\quad \Delta+n+m-3\ne0.
   \]
   \end{itemize}
   
   Let us repeat the construction
with new values of $\beta$'s to get  
$\varphi^{(i)}_1$ and $\psi^{(i)}_1$ for $i=1,2,3,4$ as following:
\begin{itemize}
   \item  For
    \[\beta_{(n,1,0)}=
    \begin{cases}
     \frac{n}{2}\alpha_{(n-1,2,1,0)};& \Delta+n-1=0,n\ge3,\\
    0; & \text{otherwise},
    \end{cases}
    \]
    then
    \begin{align*}
   (\Delta^3\varphi^{(1)}_1-D^4\psi^{(1)}_1)[v(n-1)v(2)v(1)v(0)]&=\varphi(n-1,2,1,0);\quad \Delta+n-1=0.\\
   (\Delta^3\varphi^{(1)}_1-D^4\psi^{(1)}_1)[v(n)v(m)v(p)v(q)]&=0;\quad \text{otherwise}.
   \end{align*}
   \item Also, if we put
   \[\beta_{(n,m,0)}=
    \begin{cases}
      \frac{m}{2m-4}\alpha_{(n,m-1,2,0)}; & \Delta+n+m-3=0, n\ge2,m\ge3,\\
    0; & \text{otherwise},
    \end{cases}
    \]
    we get  new $\varphi^{(2)}_1$ and $\psi^{(2)}_1$ such that 
    \begin{align*}
   (\Delta^3\varphi^{(2)}_1-D^4\psi^{(2)}_1)[v(n)v(m-1)v(2)v(0)]&=\varphi(n,m-1,2,0);\quad \Delta+n+m-3=0,\\
   (\Delta^3\varphi^{(2)}_1-D^4\psi^{(2)}_1)[v(n)v(m)v(p)v(q)]&=0;\quad \text{otherwise}.
   \end{align*}
  
    \item For  
    \[
    \beta_{(n,m,1)}=
    \begin{cases}
     -\frac{m}{2}\alpha_{(n,m-1,2,1)}; & \Delta+n+m-2=0,  n\ge2,m\ge3,\\
    0; & \text{otherwise},
    \end{cases}
    \]
    we have that
    \begin{align*}
   (\Delta^3\varphi^{(3)}_1-D^4\psi^{(3)}_1)[v(n)v(m-1)v(2)v(1)]&=\varphi(n,m-1,2,1);\quad \Delta+n+m-2=0,\\
   (\Delta^3\varphi^{(3)}_1-D^4\psi^{(3)}_1)[v(n)v(m)v(p)v(q)]&=0;\quad \text{otherwise}.
   \end{align*}
   \item For
   \[
    \beta_{(n,m,p)}=
    \begin{cases}
    -\frac{\alpha_{(n,m,p,0)}}{\Delta+n+m+p-4};& \Delta+n+m+p-4\ne0,  n,m,p\ge2,\\
    \frac{p}{-2p+4}\alpha_{(n,m,p-1,2)}; & \Delta+n+m+p-4=0,  n,m\ge2, p\ge3,\\
    0; & \text{otherwise},
    \end{cases}
    \]
we have that 
   \begin{align*}
   (\Delta^3\varphi^{(4)}_1-D^4\psi^{(4)}_1)[v(n)v(m)v(p)v(1)]=& \ \varphi(n,m,p,1);\quad \Delta+n+m+p-4\ne0,\\
   (\Delta^3\varphi^{(4)}_1-D^4\psi^{(4)}_1)[v(n)v(m)v(p-1)v(2)]=& \ \varphi(n,m,p-1,2);\\
   & \ \text{where} \ \Delta+n+m+p-4=0,\\
   (\Delta^3\varphi^{(4)}_1-D^4\psi^{(4)}_1)[v(n)v(m)v(p)v(q)]=& \ 0;\quad \text{otherwise}.
   \end{align*}
   \end{itemize}
   
Hence, every $4$-cocycle is a coboundary.

In a similar way,
    \[
 \begin{aligned}
 (\Delta^{n}\varphi - D^{n+1}\psi )[i_1|i_2|\ldots|i_{n+1}]&=\sum_{j=1}^n\frac{(-1)^ji_ji_{j+1}}{i_j+i_{j+1}-1}\Delta\alpha_{(i_1,\ldots,i_j+i_{j+1}-1,\ldots,i_{n+1})}\\
  & +\sum^n_{j=1}(-1)^{j}\frac{i_j(i_j-1)}{i_j+i_{j+1}-1}\alpha_{(i_1,\ldots,i_j+i_{j+1}-1,\ldots,i_{n+1})}\\
  &+\sum^n_{j=2}\sum^{j-1}_{t=1}(-1)^{j}\frac{i_ji_{j+1}}{i_j+i_{j+1}-1}(i_t-1)\alpha_{(i_1,\ldots,i_j+i_{j+1}-1,\ldots,i_{n+1})}\\
   &+\sum^{n+1}_{t=3}\sum^{t-2}_{j=1}(-1)^{j}i_t\frac{(i_j-1)(i_{j+1}-1)}{i_j+i_{j+1}-1}\alpha_{(i_1,\ldots ,i_j+i_{j+1},\ldots, i_t-1,\ldots ,i_{n+1})}\\
     &+\sum^n_{t=1}(-1)^{t}i_t\frac{(i_t-2)(i_{t+1}-1)}{i_t+i_{t+1}-2}\alpha_{(i_1,\ldots ,i_t+i_{t+1}-1,\ldots ,i_{n+1})}\\
     &+\sum^{n}_{t=1}(-1)^{t}i_{t+1}\frac{(i_t-1)(i_{t+1}-2)}{i_t+i_{t+1}-2}\alpha_{(i_1,\ldots ,i_t+i_{t+1}-1,\ldots ,i_{n+1})}\\
 \end{aligned}
 \]
  \[
 \begin{aligned}
 (\Delta^{n}\varphi - D^{n+1}\psi )[i_1|i_2|\ldots|i_{n-1}|1|0]&=\sum_{j=1}^{n-2}\frac{(-1)^ji_ji_{j+1}}{i_j+i_{j+1}-1}\Delta\alpha_{(i_1,\ldots,i_j+i_{j+1}-1,\ldots,i_{n-1},1,0)}\\
& +(-1)^{n-1}\Delta\alpha_{(i_1,i_2,\ldots,i_{n-1},0)}\\
  & +\sum^{n-2}_{j=1}(-1)^{j}\frac{i_j(i_j-1)}{i_j+i_{j+1}-1}\alpha_{(i_1,\ldots,i_j+i_{j+1}-1,\ldots,i_{n-1},1,0)}\\
  & +(-1)^{n-1}\alpha_{(i_1,i_2,\ldots,i_{n-1},0)}\\
  &+\sum^{n-2}_{j=1}\sum^{j-1}_{t=1}(-1)^{j}\frac{i_ji_{j+1}}{i_j+i_{j+1}-1}(i_t-1)\\
  &\times\alpha_{(i_1,\ldots,i_j+i_{j+1}-1,\ldots,i_{n-1},1,0)}\\
  &+\sum^{n-1}_{t=1}(-1)^{n-1}(i_t-1)\alpha_{(i_1,i_2,\ldots,i_{n-1},0)}\\
  &+\sum^{n-1}_{t=3}\sum^{t-2}_{j=1}(-1)^{j}i_t\frac{(i_j-1)(i_{j+1}-1)}{i_j+i_{j+1}-1}\\
   &\times\alpha_{(i_1,\ldots ,i_j+i_{j+1},\ldots, i_t-1,\ldots ,i_{n-1},1,0)}\\
   &+\sum^{n-2}_{t=1}(-1)^{t}i_t\frac{(i_t-2)(i_{t+1}-1)}{i_t+i_{t+1}-2}\alpha_{(i_1,\ldots ,i_t+i_{t+1}-1,\ldots ,i_{n-1},1,0)}\\
   \end{aligned}
   \]
   \[
   \begin{aligned}
   &+\sum^{n-2}_{t=1}(-1)^{t}i_{t+1}\frac{(i_t-1)(i_{t+1}-2)}{i_t+i_{t+1}-2}\alpha_{(i_1,\ldots ,i_t+i_{t+1}-1,\ldots ,i_{n-1},1,0)}.
 \end{aligned}
 \]
 
Hence, $\varphi -D^{n+1}\tilde \C^{n+1}$ is a $n$-cocycle in $\C^{n}$ for various $\Delta$:

{\sc Case 1:} $\Delta+i_1+\ldots+i_n-n-1 \ne0$;  $i_1,\ldots,i_n\ge2.$
 \[
 \begin{aligned}
(-1)^{n+1} (\Delta+i_1+\ldots+i_n-n-1 )\alpha_{(i_1,i_2,\ldots,i_n)}&=\sum_{j=1}^{n-1}\frac{(-1)^ji_ji_{j+1}}{i_j+i_{j+1}-1}\Delta\alpha_{(i_1,\ldots,i_j+i_{j+1}-1,\ldots,i_{n},1)}\\
  & +\sum^{n-1}_{j=1}(-1)^{j}\frac{i_j(i_j-1)}{i_j+i_{j+1}-1}\alpha_{(i_1,\ldots,i_j+i_{j+1}-1,\ldots,i_n,1)}\\
  &+\sum^n_{j=2}\sum^{j-1}_{t=1}(-1)^{j}\frac{i_ji_{j+1}}{i_j+i_{j+1}-1}(i_t-1)\\
&\times\alpha_{(i_1,\ldots,i_j+i_{j+1}-1,\ldots,i_n,1)}\\
   &+\sum^{n-1}_{t=3}\sum^{t-2}_{j=1}(-1)^{j}i_t\frac{(i_j-1)(i_{j+1}-1)}{i_j+i_{j+1}-1}\\
   &\times\alpha_{(i_1,\ldots ,i_j+i_{j+1},\ldots, i_t-1,\ldots ,i_n,1)}\\
     &+\sum^n_{t=1}(-1)^{t}i_t\frac{(i_t-2)(i_{t+1}-1)}{i_t+i_{t+1}-2}\\
     &\times\alpha_{(i_1,\ldots ,i_t+i_{t+1}-1,\ldots ,i_n,1)}\\
     &+\sum^{n}_{t=1}(-1)^{t}i_{t+1}\frac{(i_t-1)(i_{t+1}-2)}{i_t+i_{t+1}-2}\\
     &\times\alpha_{(i_1,\ldots ,i_t+i_{t+1}-1,\ldots ,i_n,1)}.\\
 \end{aligned}
 \]
 
 {\sc Case 2:} $\Delta+i_1+\ldots+i_{n-1}-n \ne0$;  $i_1,\ldots,i_{n-1}\ge2.$
  \[
 \begin{aligned}
(-1)^n (\Delta+i_1+\ldots+i_{n-1}-n )\alpha_{(i_1,i_2,\ldots,i_{n-1},0)}&=\sum_{j=1}^{n-2}\frac{(-1)^ji_ji_{j+1}}{i_j+i_{j+1}-1}\Delta\alpha_{(i_1,\ldots,i_j+i_{j+1}-1,\ldots,i_{n-1},1,0)}\\
  & +\sum^{n-2}_{j=1}(-1)^{j}\frac{i_j(i_j-1)}{i_j+i_{j+1}-1}\\
  &\times\alpha_{(i_1,\ldots,i_j+i_{j+1}-1,\ldots,i_{n-1},1,0)}\\
  &+\sum^{n-2}_{j=1}\sum^{j-1}_{t=1}(-1)^{j}\frac{i_ji_{j+1}}{i_j+i_{j+1}-1}(i_t-1)\\
  &\times\alpha_{(i_1,\ldots,i_j+i_{j+1}-1,\ldots,i_{n-1},1,0)}\\
  \end{aligned}
  \]
  \[
  \begin{aligned}
   &+\sum^{n-1}_{t=3}\sum^{t-2}_{j=1}(-1)^{j}i_t\frac{(i_j-1)(i_{j+1}-1)}{i_j+i_{j+1}-1}\alpha_{(i_1,\ldots ,i_j+i_{j+1},\ldots, i_t-1,\ldots ,i_{n-1},1,0)}\\
        &+\sum^{n-2}_{t=1}(-1)^{t}i_t\frac{(i_t-2)(i_{t+1}-1)}{i_t+i_{t+1}-2}\alpha_{(i_1,\ldots ,i_t+i_{t+1}-1,\ldots ,i_{n-1},1,0)}\\
     &+\sum^{n-2}_{t=1}(-1)^{t}i_{t+1}\frac{(i_t-1)(i_{t+1}-2)}{i_t+i_{t+1}-2}\alpha_{(i_1,\ldots ,i_t+i_{t+1}-1,\ldots ,i_{n-1},1,0)}
 \end{aligned}
 \]
 
 {\sc Case 3:} $\Delta+i_1+\ldots+i_{n-1}-n =0$;  $i_1,\ldots,i_{n-1}\ge2.$
  \[
 \begin{aligned}
 \frac{2(-1)^n}{i_{n-2}}\alpha_{(i_1,i_2,\ldots,i_{n-2},1,0)}&=\sum_{j=1}^{n-4}\frac{(-1)^ji_ji_{j+1}}{i_j+i_{j+1}-1}\Delta\alpha_{(i_1,\ldots,i_j+i_{j+1}-1,\ldots,i_{n-2}-1,2,1,0)}\\
&+\frac{(-1)^{n-3}i_{n-3}(i_{n-2}-1)}{i_{n-3}+i_{n-2}-2}\Delta\alpha_{(i_1,\ldots,i_{n-4},i_{n-3}+i_{n-2}-2,2,1,0)}\\
  & +\sum^{n-4}_{j=1}(-1)^{j}\frac{i_j(i_j-1)}{i_j+i_{j+1}-1}\alpha_{(i_1,\ldots,i_j+i_{j+1}-1,\ldots,i_{n-2}-1,2,1,0)}\\
  &+(-1)^{n-3}\frac{i_{n-3}(i_{n-3}-1)}{i_{n-3}+i_{n-2}-2}\alpha_{(i_1,\ldots,i_{n-3}+i_{n-2}-2,2,1,0)}\\
  &+\sum^{n-5}_{t=1}\sum^{t+1}_{j=2}(-1)^{j}\frac{i_ji_{j+1}}{i_j+i_{j+1}-1}(i_t-1)\alpha_{(i_1,\ldots,i_j+i_{j+1}-1,\ldots,i_{n-2}-1,2,1,0)}\\
  &+\sum^{n-4}_{t=1}(-1)^{n-3}\frac{i_{n-3}(i_{n-2}-1)}{i_{n-3}+i_{n-2}-2}(i_t-1)\alpha_{(i_1,\ldots,i_{n-3}+i_{n-2}-2,2,1,0)}\\
   &+\sum^{n-4}_{t=3}\sum^{t-2}_{j=1}(-1)^{j}i_t\frac{(i_j-1)(i_{j+1}-1)}{i_j+i_{j+1}-1}\alpha_{(i_1,\ldots ,i_j+i_{j+1},\ldots, i_t-1,\ldots ,i_{n-2},2,1,0)}\\
   &+(-1)^{n-3}\frac{2(i_{n-3}-1)(i_{n-2}-2)}{i_{n-3}+i_{n-2}-2}\alpha_{(i_1,\ldots ,i_j+i_{j+1},\ldots, i_t-1,\ldots ,i_{n-2}-1,2,1,0)}\\
        &+\sum^{n-4}_{t=1}(-1)^{t}i_t\frac{(i_t-2)(i_{t+1}-1)}{i_t+i_{t+1}-2}\alpha_{(i_1,\ldots ,i_t+i_{t+1}-1,\ldots ,i_{n-2}-1,2,1,0)}\\
        &+(-1)^{n-3}i_{n-3}\frac{(i_{n-3}-2)(i_{n-2}-2)}{i_{n-3}+i_{n-2}-3}\alpha_{(i_1,\ldots ,i_{n-3}+i_{n-2}-2,2,1,0)}\\
     &+\sum^{n-4}_{t=1}(-1)^{t}i_{t+1}\frac{(i_t-1)(i_{t+1}-2)}{i_t+i_{t+1}-2}\alpha_{(i_1,\ldots ,i_t+i_{t+1}-1,\ldots ,i_{n-2}-1,2,1,0)}\\
     &+(-1)^{n-3}\frac{(i_{n-2}-1)(i_{n-3}-1)(i_{n-2}-3)}{i_{n-3}+i_{n-2}-3}\alpha_{(i_1,\ldots  ,i_{n-3}+i_{n-2}-2,2,1,0)},
 \end{aligned}
 \]
  \[
 \begin{aligned}
 (-1)^n\frac{2i_{n-1}-4}{i_{n-1}}\alpha_{(i_1,i_2,\ldots,i_{n-1},0)}&=\sum_{j=1}^{n-3}\frac{(-1)^ji_ji_{j+1}}{i_j+i_{j+1}-1}\Delta\alpha_{(i_1,\ldots,i_j+i_{j+1}-1,\ldots,i_{n-1}-1,2,0)}\\
 &+\frac{(-1)^{n-2}i_{n-2}(i_{n-1}-1)}{i_{n-2}+i_{n-1}-2}\Delta\alpha_{(i_1,\ldots,i_{n-3},i_{n-2}+i_{n-1}-2,2,0)}\\
  & +\sum^{n-3}_{j=1}(-1)^{j}\frac{i_j(i_j-1)}{i_j+i_{j+1}-1}\alpha_{(i_1,\ldots,i_j+i_{j+1}-1,\ldots,i_{n-1}-1,2,0)}\\
  &+(-1)^{n-2}\frac{i_{n-2}(i_{n-2}-1)}{i_{n-2}+i_{n-1}-2}\alpha_{(i_1,\ldots,i_{n-3},i_{n-2}+i_{n-1}-2,2,0)}\\
  &+\sum^{n-4}_{t=1}\sum^{t+1}_{j=2}(-1)^{j}\frac{i_ji_{j+1}}{i_j+i_{j+1}-1}(i_t-1)\alpha_{(i_1,\ldots,i_j+i_{j+1}-1,\ldots,i_{n-1}-1,2,0)}\\
  &+\sum^{n-3}_{t=1}(-1)^{n-2}\frac{i_{n-2}(i_{n-1}-1)}{i_{n-2}+i_{n-1}-2}(i_t-1)\alpha_{(i_1,\ldots,i_{n-2}+i_{n-1}-2,2,0)}\\
    &+\sum^{n-3}_{j=1}(-1)^{j}i_j\frac{(i_j-2)(i_{j+1}-1)}{i_j+i_{j+1}-2}\alpha_{(i_1,\ldots ,i_j+i_{j+1}-1,\ldots ,i_{n-1}-1,2,0)}\\
    &+(-1)^{n-2}i_{n-2}\frac{(i_{n-2}-2)(i_{n-1}-2)}{i_{n-2}+i_{n-1}-3}\alpha_{(i_1,\ldots ,i_{n-2}+i_{n-1}-2,2,0)}\\
&+\sum^{n-3}_{j=1}(-1)^{j}i_{j+1}\frac{(i_j-1)(i_{j+1}-2)}{i_j+i_{j+1}-2}\alpha_{(i_1,\ldots ,i_t+i_{t+1}-1,\ldots ,i_{n-1}-1,2,0)}\\
&+(-1)^{n-2}\frac{(i_{n-1}-1)(i_{(n-2)}-1)(i_{n-1}-3)}{i_{n-2}+i_{n-1}-3}\\
&\times\alpha_{(i_1,\ldots ,i_{n-2}+i_{n-1}-2,2,0)}\\
   &+\sum^{n-4}_{j=1}\sum^{n-2}_{t=j+2}(-1)^{j}i_t\frac{(i_j-1)(i_{j+1}-1)}{i_j+i_{j+1}-1}\\
   &\times\alpha_{(i_1,\ldots ,i_j+i_{j+1},\ldots, i_t-1,\ldots ,i_{n-1}-1,2,0)}\\
   &+\sum^{n-3}_{j=1}(-1)^{j}\frac{(i_{n-1}-1)(i_j-1)(i_{j+1}-1)}{i_j+i_{j+1}-1}\\
   &\times\alpha_{(i_1,\ldots ,i_j+i_{j+1},\ldots ,i_{n-1}-2,2,0)}\\
   &+\sum^{n-3}_{j=1}(-1)^{j}\frac{2(i_j-1)(i_{j+1}-1)}{i_j+i_{j+1}-1}\alpha_{(i_1,\ldots ,i_j+i_{j+1},\ldots, i_t-1,\ldots ,i_{n-1}-1,1,0)}\\
   &+(-1)^{n-2}\frac{2(i_{n-2}-1)(i_{n-1}-2)}{i_{n-2}+i_{n-1}-3}\alpha_{(i_1,\ldots ,i_{n-2}+i_{n-1}-1,1,0)}\\
    \end{aligned}
 \]
  \[
 \begin{aligned}
- \frac{2(-1)^n}{i_{n-1}}\alpha_{(i_1,i_2,\ldots,i_{n-1},1)}&=\sum_{j=1}^{n-3}\frac{(-1)^ji_ji_{j+1}}{i_j+i_{j+1}-1}\Delta\alpha_{(i_1,\ldots,i_j+i_{j+1}-1,\ldots,i_{n-1}-1,2,1)}\\
 &+\frac{(-1)^{n-2}i_{n-2}(i_{n-1}-1)}{i_{n-2}+i_{n-1}-2}\Delta\alpha_{(i_1,\ldots,i_{n-3},i_{n-2}+i_{n-1}-2,2,1)}\\
  & +\sum^{n-3}_{j=1}(-1)^{j}\frac{i_j(i_j-1)}{i_j+i_{j+1}-1}\alpha_{(i_1,\ldots,i_j+i_{j+1}-1,\ldots,i_{n-1}-1,2,1)}\\
  &+(-1)^{n-2}\frac{i_{n-2}(i_{n-2}-1)}{i_{n-2}+i_{n-1}-2}\alpha_{(i_1,\ldots,i_{n-3},i_{n-2}+i_{n-1}-2,2,1)}\\
  &+\sum^{n-4}_{t=1}\sum^{t+1}_{j=2}(-1)^{j}\frac{i_ji_{j+1}}{i_j+i_{j+1}-1}(i_t-1)\alpha_{(i_1,\ldots,i_j+i_{j+1}-1,\ldots,i_{n-1}-1,2,1)}\\
  &+\sum^{n-3}_{t=1}(-1)^{n-2}\frac{i_{n-2}(i_{n-1}-1)}{i_{n-2}+i_{n-1}-2}(i_t-1)\alpha_{(i_1,\ldots,i_{n-2}+i_{n-1}-2,2,1)}\\
    &+\sum^{n-3}_{j=1}(-1)^{j}i_j\frac{(i_j-2)(i_{j+1}-1)}{i_j+i_{j+1}-2}\alpha_{(i_1,\ldots ,i_j+i_{j+1}-1,\ldots ,i_{n-1}-1,2,1)}\\
    &+(-1)^{n-2}i_{n-2}\frac{(i_{n-2}-2)(i_{n-1}-2)}{i_{n-2}+i_{n-1}-3}\alpha_{(i_1,\ldots ,i_{n-2}+i_{n-1}-2,2,1)}\\
&+\sum^{n-3}_{j=1}(-1)^{j}i_{j+1}\frac{(i_j-1)(i_{j+1}-2)}{i_j+i_{j+1}-2}\alpha_{(i_1,\ldots ,i_t+i_{t+1}-1,\ldots ,i_{n-1}-1,2,1)}\\
&+(-1)^{n-2}\frac{(i_{n-1}-1)(i_{(n-2)}-1)(i_{n-1}-3)}{i_{n-2}+i_{n-1}-3}\alpha_{(i_1,\ldots ,i_{n-2}+i_{n-1}-2,2,1)}\\
   &+\sum^{n-4}_{j=1}\sum^{n-2}_{t=j+2}(-1)^{j}i_t\frac{(i_j-1)(i_{j+1}-1)}{i_j+i_{j+1}-1}\alpha_{(i_1,\ldots ,i_j+i_{j+1},\ldots, i_t-1,\ldots ,i_{n-1}-1,2,1)}\\
   &+\sum^{n-3}_{j=1}(-1)^{j}\frac{(i_{n-1}-1)(i_j-1)(i_{j+1}-1)}{i_j+i_{j+1}-1}\alpha_{(i_1,\ldots ,i_j+i_{j+1},\ldots ,i_{n-1}-2,2,1)}\\
   &+\sum^{n-3}_{j=1}(-1)^{j}\frac{(i_j-1)(i_{j+1}-1)}{i_j+i_{j+1}-1}\alpha_{(i_1,\ldots ,i_j+i_{j+1},\ldots, i_t-1,\ldots ,i_{n-1}-1,2,0)}\\
   &+(-1)^{n-2}\frac{(i_{n-2}-1)(i_{n-1}-2)}{i_{n-2}+i_{n-1}-3}\alpha_{(i_1,\ldots ,i_{n-2}+i_{n-1}-1,2,0)}\\
   &+(-1)^{n-1}\frac{i_{n-1}-2}{i_{n-1}}\alpha_{(i_1,\ldots ,i_{n-1}+1,0)}.
 \end{aligned}
 \]
 \[
 \begin{aligned}
(-1)^{n+1}\frac{2i_n-4}{i_{n}}\alpha_{(i_1,i_2,\ldots,i_{n})}&=\sum_{j=1}^{n-2}\frac{(-1)^ji_ji_{j+1}}{i_j+i_{j+1}-1}\Delta\alpha_{(i_1,\ldots,i_j+i_{j+1}-1,\ldots,i_{n}-1,2)}\\
&+\frac{(-1)^{n-1}i_{n-1}(i_{n}-1)}{i_{n-1}+i_{n}-2}\Delta\alpha_{(i_1,\ldots,i_{n-1}+i_{n}-2,2)}\\
  & +\sum^{n-2}_{j=1}(-1)^{j}\frac{i_j(i_j-1)}{i_j+i_{j+1}-1}\alpha_{(i_1,\ldots,i_j+i_{j+1}-1,\ldots,i_{n}-1,2)}\\
   &+(-1)^{n-1}\frac{i_{n-1}(i_{n-1}-1)}{i_{n-1}+i_{n}-2}\alpha_{(i_1,\ldots,i_{n-1}+i_{n}-2,2)}\\
  &+\sum^{n-3}_{t=1}\sum^{t+1}_{j=2}(-1)^{j}\frac{i_ji_{j+1}}{i_j+i_{j+1}-1}(i_t-1)\alpha_{(i_1,\ldots,i_j+i_{j+1}-1,\ldots,i_{n}-1,2)}\\
   &+\sum^{n-2}_{t=1}(-1)^{n-1}\frac{i_{n-1}(i_{n}-1)}{i_{n-1}+i_{n}-2}(i_t-1)\alpha_{(i_1,\ldots,i_{n-1}+i_{n}-2,2)}\\
    &+\sum^{n-2}_{j=1}(-1)^{j}i_j\frac{(i_j-2)(i_{j+1}-1)}{i_j+i_{j+1}-2}\alpha_{(i_1,\ldots ,i_t+i_{t+1}-1,\ldots ,i_{n}-1,2)}\\
    &+(-1)^{n-1}\frac{i_n(i_{n-1}-2)(i_{n}-2)}{i_{n-1}+i_{n}-3}\alpha_{(i_1,\ldots , ,i_{n-1}+i_{n}-2,2)}\\
    &+\sum^{n-2}_{j=1}(-1)^{j}i_{j+1}\frac{(i_j-1)(i_{j+1}-2)}{i_j+i_{j+1}-2}\alpha_{(i_1,\ldots ,i_t+i_{t+1}-1,\ldots ,i_{n}-1,2)}\\
    &+(-1)^{n-1}\frac{(i_n-1)(i_{n-1}-1)(i_{n}-3)}{i_{n-1}+i_{n}-3}\alpha_{(i_1,\ldots , ,i_{n-1}+i_{n}-2,2)}\\
       &+\sum^{n-3}_{j=1}\sum^{n-1}_{t=j+1}(-1)^{j}i_t\frac{(i_j-1)(i_{j+1}-1)}{i_j+i_{j+1}-1}\alpha_{(i_1,\ldots ,i_j+i_{j+1},\ldots, i_t-1,\ldots ,i_n-1,2)}\\
          &+\sum^{n-2}_{j=1}(-1)^{j}\frac{(i_n-1)(i_j-1)(i_{j+1}-1)}{i_j+i_{j+1}-1}\alpha_{(i_1,\ldots ,i_j+i_{j+1},\ldots ,i_n-1,2)}\\
           &+\sum^{n-1}_{j=1}(-1)^{j}\frac{2(i_j-1)(i_{j+1}-1)}{i_j+i_{j+1}-1}\alpha_{(i_1,\ldots ,i_j+i_{j+1},\ldots ,i_n-1,1)}\\
           &+(-1)^{n-1}\frac{2(i_{n-1}-1)(i_{n}-2)}{i_{n-1}+i_{n}-2}\alpha_{(i_1,\ldots ,i_{n-1}+i_n-1,1)}.
 \end{aligned}
 \]
 
  Let $\varphi_1\in \tilde \C^{n-1}, \psi_1\in \tilde \C^n$, such that
     \begin{align*}
     \varphi_1(i_1,i_2,\ldots,i_{n-1})&=\beta_{(i_1,i_2,\ldots,i_{n-1})},\\
     \psi_1(i_1,i_2,\ldots,i_{n})&=\sum^{n-1}_{j=1}(-1)^{j}\frac{(i_j-1)(i_{j+1}-1)}{i_j+i_{j+1}-1}\beta_{(i_1,\ldots,i_j+i_{j+1},\ldots,i_n)},
     \end{align*}
      \[
      \begin{aligned}
     \psi_1(i_1,i_2,\ldots,i_{n-2},1,0){}&=\sum^{n-2}_{j=1}(-1)^{j}\frac{(i_j-1)(i_{j+1}-1)}{i_j+i_{j+1}-1}\beta_{(i_1,\ldots,i_j+i_{j+1},\ldots,i_{n-2},1,0)}\\
     & +(-1)^{n-1}\beta_{(i_1,i_2,\ldots,i_{n-1},1)}.
     \end{aligned}
     \]
     
        Namely,
         \[
 \begin{aligned}
 (\Delta^{n-1}\varphi_1 - D^n\psi_1 )[i_1|i_2|\ldots|i_n]&=\sum_{j=1}^{n-1}\frac{(-1)^ji_ji_{j+1}}{i_j+i_{j+1}-1}\Delta\beta_{(i_1,\ldots,i_j+i_{j+1}-1,\ldots,i_n)}\\
  & +\sum^{n-1}_{j=1}(-1)^{j}\frac{i_j(i_j-1)}{i_j+i_{j+1}-1}\beta_{(i_1,\ldots,i_j+i_{j+1}-1,\ldots,i_n)}\\
  &+\sum^{n-1}_{j=2}\sum^{j-1}_{t=1}(-1)^{j}\frac{i_ji_{j+1}}{i_j+i_{j+1}-1}(i_t-1)\beta_{(i_1,\ldots,i_j+i_{j+1}-1,\ldots,i_n)}\\
   &+\sum^{n}_{t=3}\sum^{t-2}_{j=1}(-1)^{j}i_t\frac{(i_j-1)(i_{j+1}-1)}{i_j+i_{j+1}-1}\beta_{(i_1,\ldots ,i_j+i_{j+1},\ldots, i_t-1,\ldots ,i_n)}\\
     &+\sum^{n-1}_{t=1}(-1)^{t}i_t\frac{(i_t-2)(i_{t+1}-1)}{i_t+i_{t+1}-2}\beta_{(i_1,\ldots ,i_t+i_{t+1}-1,\ldots ,i_n)}\\
     &+\sum^{n-1}_{t=1}(-1)^{t}i_{t+1}\frac{(i_t-1)(i_{t+1}-2)}{i_t+i_{t+1}-2}\beta_{(i_1,\ldots ,i_t+i_{t+1}-1,\ldots ,i_n)},\\
 \end{aligned}
 \]
  \[
 \begin{aligned}
 (\Delta^{n-1}\varphi_1 - D^n\psi_1 )[i_1|i_2|\ldots|i_{n-2}|1|0]&=\sum_{j=1}^{n-2}\frac{(-1)^ji_ji_{j+1}}{i_j+i_{j+1}-1}\Delta\beta_{(i_1,\ldots,i_j+i_{j+1}-1,\ldots,i_{n-2},1,0)}\\
  & +\sum^{n-2}_{j=1}(-1)^{j}\frac{i_j(i_j-1)}{i_j+i_{j+1}-1}\beta_{(i_1,\ldots,i_j+i_{j+1}-1,\ldots,i_{n-2},1,0)}\\
  &+\sum^{n-2}_{j=1}\sum^{j-1}_{t=1}(-1)^{j}\frac{i_ji_{j+1}}{i_j+i_{j+1}-1}(i_t-1)\\
  &\times\beta_{(i_1,\ldots,i_j+i_{j+1}-1,\ldots,i_{n-2},1,0)}\\
  &+(-1)^{n-1}(\Delta+i_1+\ldots+i_{n-2}-n+1)\beta_{(i_1,i_2,\ldots,i_{n-2},0)}\\
   &+\sum^{n-1}_{t=3}\sum^{t-2}_{j=1}(-1)^{j}i_t\frac{(i_j-1)(i_{j+1}-1)}{i_j+i_{j+1}-1}\\
   &\times\beta_{(i_1,\ldots ,i_j+i_{j+1},\ldots, i_t-1,\ldots ,i_{n-2},1,0)}\\
   &+\sum^{n-2}_{t=1}(-1)^{t}i_t\frac{(i_t-2)(i_{t+1}-1)}{i_t+i_{t+1}-2}\beta_{(i_1,\ldots ,i_t+i_{t+1}-1,\ldots ,i_{n-2},1,0)}\\
   \end{aligned}
   \]
   \[
   \begin{aligned}
     &+\sum^{n-2}_{t=1}(-1)^{t}i_{t+1}\frac{(i_t-1)(i_{t+1}-2)}{i_t+i_{t+1}-2}\beta_{(i_1,\ldots ,i_t+i_{t+1}-1,\ldots ,i_{n-2},1,0)},
 \end{aligned}
 \]
where 
  \begin{itemize}
        \item for
    \[\beta_{(i_1,\ldots,i_{n-2},0)}=
    \begin{cases}
    \frac{\alpha_{(i_1,\ldots,i_{n-2},1,0)}}{\Delta+i_1+\ldots+i_{n-2}-n+1};& \Delta+i_1+\ldots+i_{n-2}-n+1\ne0,i_1,\ldots,i_{n-2}\ge2,\\
    0; & \text{otherwise},
    \end{cases}
    \]
 then $\Delta^{n-1}\varphi_1 $ is a coboundary in $\tilde \C^n$, and
   \begin{align*}
   (\Delta^{n-1}\varphi_1-D^n\psi_1)[v(i_1)\ldots v(i_{n-2})v(1)v(0)]=& \ \varphi(i_1,\ldots,i_{n-2},1,0);\\
   &\text{where} \ \Delta+i_1+\ldots+i_{n-2}-n+1\ne0.
   \end{align*}
   \end{itemize}
   
Let us repeat the construction
with new values of $\beta$'s to get  
$\varphi^{(i)}_1$ and $\psi^{(i)}_1$ for $i=1,2,3,4$ as following:
\begin{itemize}
   \item For 
    \[\beta_{(i_1,\ldots,i_{n-3},1,0)}=
    \begin{cases}
     \frac{(-1)^ni_1}{2}\alpha_{(i_1-1,\ldots,i_{n-3},2,1,0)};& \Delta+i_1+\ldots+i_{n-3}-n+3=0,i_1\ge3,\\
    0; & \text{otherwise},
    \end{cases}
    \]
    we have
    \begin{align*}
   (\Delta^{n-1}\varphi^{(1)}_1-D^n\psi^{(1)}_1)[v(i_1-1)\ldots v(i_{n-3})v(2)v(1)v(0)]=& \ \varphi(i_1-1,\ldots,i_{n-3},2,1,0);\\
   \text{where} \ \Delta+i_1+\ldots+i_{n-3}-n+3=& \ 0,\\
   (\Delta^{n-1}\varphi^{(1)}_1-D^n\psi^{(1)}_1)[v(i_1)\ldots v(i_n)]=& \ 0;\quad \text{otherwise}.
   \end{align*}
   \item Also if we put
   \[\beta_{(i_1,\ldots,i_{n-2},0)}=
    \begin{cases}
      \frac{(-1)^ni_{n-2}}{2i_{n-2}-4}\alpha_{(i_1,\ldots,i_{n-2}-1,2,0)}; & \Delta+i_{1}+\ldots+i_{n-2}-n+1=0, i_{n-2}\ge 3,\\
    0; & \text{otherwise},
    \end{cases}
    \]
    we get  new $\varphi^{(2)}_1$ and $\psi^{(2)}_1$ such that 
    \begin{align*}
   (\Delta^{n-1}\varphi^{(2)}_1-D^n\psi^{(2)}_1)[v(i_1)\ldots v(i_{n-2}-1)v(2)v(0)]=& \ \varphi(i_1,\dots,i_{n-2}-1,2,0);\\
   \text{where} \ \Delta+i_1+\ldots+i_{n-2}-n+1=& \ 0,\\
   (\Delta^{n-1}\varphi^{(2)}_1-D^n\psi^{(2)}_1)[v(i_1)\ldots v(i_n)]=& \ 0;\quad \text{otherwise}.
   \end{align*}
    \item  For 
\[
    \beta_{(i_1,\ldots,i_{n-2},1)}=
    \begin{cases}
     \frac{(-1)^{n-1}i_{n-2}}{2}\alpha_{(i_1,\ldots,i_{n-2}-1,2,1)}; & \Delta+i_{1}+\ldots+i_{n-2}-n+2=0 \ \mbox{and}\\
    &  i_{n-2}\ge3,\\
    0; & \text{otherwise},
    \end{cases}
    \]
    we have that
     \begin{align*}
   (\Delta^{n-1}\varphi^{(3)}_1-D^n\psi^{(3)}_1)[v(i_1)\ldots v(i_{n-2}-1)v(2)v(1)]=& \ \varphi(i_1,\dots,i_{n-2}-1,2,1);\\
   \text{where} \ \Delta+i_1+\ldots+i_{n-2}-n+2=& \ 0,\\
   (\Delta^{n-1}\varphi^{(3)}_1-D^n\psi^{(3)}_1)[v(i_1)\ldots v(i_n)]=& \ 0;\quad \text{otherwise}.
   \end{align*}
   \item Also if
   \[
    \beta_{(i_1,\ldots,i_{n-1})}=
    \begin{cases}
    \frac{(-1)^{n-1}\alpha_{(i_1,\ldots,i_{n-1},0)}}{\Delta+i_1+\ldots+i_{n-1}-n};& \Delta+i_1+\ldots+i_{n-1}-n\ne0,\\
    \frac{(-1)^{n-1}i_{n-1}}{2i_{n-1}-4}\alpha_{(i_1,\ldots,i_{n-2},i_{n-1}-1,2)}; & \Delta+i_1+\ldots+i_{n-1}-n=0, i_{n-1}\ge3,\\
    0; & \text{otherwise},
    \end{cases}
    \]
    we have
   \begin{align*}
   (\Delta^{n-1}\varphi^{(4)}_1-D^n\psi^{(4)}_1)[v(i_1)\ldots v(i_{n-1})v(1)]=& \ \varphi(i_1,\ldots,i_{n-1},1);\\
   \text{where} \ \Delta+i_1+\ldots+i_{n-1}-n\ne & \ 0,\\
   (\Delta^{n-1}\varphi^{(4)}_1-D^n\psi^{(4)}_1)[v(i_1)\ldots v(i_{n-1}-1)v(2)]=& \ \varphi(i_1,\ldots,i_{n-1}-1,2);\\
   \text{where} \ \Delta+i_1+\ldots+i_{n-1}-n=& \ 0\\
   (\Delta^{n-1}\varphi^{(4)}_1-D^n\psi^{(4)}_1)[v(i_1)\ldots v(i_n)]=& \ 0; \quad \text{otherwise}.
   \end{align*}
   \end{itemize}
   Hence, every $n$-cocycle is a coboundary.
   \end{proof}
  \begin{theorem} 
  For a conformal module $M_{(0,0)}$ over the 
associative conformal algebra $U(2)$ we have
\[ 
\dim_\Bbbk\Homol^n(U(3), M_{(0,0)})=
\begin{cases}
1, & n=1,\\
2, & n=2,\\
1, & n=3,\\
0, & n\ge4.
\end{cases}
\]
\end{theorem}
    
\begin{proof}
We have proven in Theorem \ref{them:H^1} that
\begin{align*}
(\Delta^1\varphi - D^2\psi )[1|0] &= -\alpha_0,\\
(\Delta^1\varphi - D^2\psi )[n|1] &= -(n-2)\alpha_n, \quad n\ge 2.
\end{align*}

Therefore, if $\Delta^1\varphi -D^1\psi$ is a cocycle in $\C^1$ then $\alpha_n=0$ for all $n\ge 0$
except, maybe, for $n=1,2$.

Coboundary cocycles in $\tilde \C^1$ are given by 
$\Delta^0 h$, where $h\in \Hom_\Lambda (\Lambda, M)$. 
Modulo $D^0\tilde C^0$, we may assume $h(1) = \beta u$, $\beta=-\alpha_1 \in \Bbbk $.
Then 
$(\Delta^0h)[n] = v(n)\beta u$. 
Choose $\psi \in \tilde \C^1$ such that 
$\psi[0]=\beta  u$ and 
$\psi[n] = 0$ for $n\ge 1$. 
Then 
\[
(\Delta^0 h - D^1\psi )[n]
=\begin{cases}
 0, & n=0, \\
 -\beta u, & n=1, \\
 0, & n\ge 2,
\end{cases}
\]
and cocycles are determined by $\alpha_2.$

We have proven in Theorem \ref{them:H^2} that:
    \begin{align*}
  (\Delta^2\varphi-D^3\psi)(n,m,p)&=-\frac{n(n-1)}{n+m-1}\alpha_{(n+m-1,p)}
    +\frac{(n-1)mp}{m+p-1}\alpha_{(n,m+p-1)}\\
    &+\frac{m(m-1)}{m+p-1}\alpha_{(n,m+p-1)}
 -\frac{n(n-2)(m-1)}{n+m-2}\alpha_{(n+m-1,p)}\\
 &-\frac{m(n-1)(m-2)}{n+m-2}\alpha_{(n+m-1,p)}+\frac{m(m-2)(p-1)}{m+p-2}\alpha_{(n,m+p-1)}\\
  &-\frac{p(n-1)(m-1)}{n+m-1}\alpha_{(n+m,p-1)}+\frac{p(m-1)(p-2)}{m+p-2}\alpha_{(n,m+p-1)},\\ 
    (\Delta^2\varphi-D^3\psi)(n,1,0)&=-(n-2) \alpha_{(n,0)}.
    \end{align*}
    
Hence, $\varphi -D^3\tilde \C^3$ is a $2$-cocycle in $\C^2$
if and only if    
   \begin{align*}
   \alpha_{(n,0)}&=0;\quad n\ge1; \quad n\ne2,\\
  -(n+m-3)\alpha_{(n,m)}&=-\frac{n(n-1)}{n+m-1}\alpha_{(n+m-1,1)}
 -\frac{n(n-2)(m-1)}{n+m-2}\alpha_{(n+m-1,1)}\\
 &-\frac{m(n-1)(m-2)}{n+m-2}\alpha_{(n+m-1,1)}.
   \end{align*} 

Therefore, cocycles in $\C^2$ are determined by $\alpha_{(n+m-1,1)}$ , $\alpha_{(2,1)}$ and $\alpha_{(2,0)}$ for all $n,m\ge2$.
  
  Choose 
\[
\varphi_1\in \tilde C^1,\quad \psi_1 \in \tilde C^2 = \Hom_\Lambda (\mathrm A_2, M),
\]
such that 
\begin{align*}
\varphi_1[n] =\beta_n,\quad \psi_1[1|0] =& \ \beta_1, \quad \psi_1[n|m] = \dfrac{(n-1)(m-1)}{n+m-1} \beta_{n+m},\\
(\Delta^1\varphi_1-D^2\psi_1)[n|m]=& 
 - \dfrac{n(n-1)}{n+m-1} \beta_{n+m-1}
-\dfrac{n(n-2)(m-1)}{n+m-2} \beta_{n+m-1}\\
&-\dfrac{m(n-1)(m-2)}{n+m-2} \beta_{n+m-1}\\
(\Delta^1\varphi_1 - D^2\psi_1 )[1|0] =& -\beta_0,
\end{align*}
 where
   \[
    \beta_n=-\dfrac{\alpha_{(n,1)}}{n-2};\quad n\ge3.
    \]
    
Then $\Delta^1\varphi_1$ is coboundary in $\tilde \C^2$ and
\begin{align*}
(\Delta^1\varphi_1 - D^2\psi_1 )[n|1] &= (-n+2)\beta_n=\alpha_{(n,1)}=\varphi(n,1);\quad n\ge3,\\
(\Delta^1\varphi_1 - D^2\psi_1 )[2|1]&=0,\\
(\Delta^1\varphi_1 - D^2\psi_1 )[2|0]&=0.
\end{align*}

Hence, 
$\varphi -\Delta^1\varphi_1 \in D^2\tilde \C^2$, so 
every cocycle in $\C^2$ is a coboundary or zero except $\alpha_{(2,0)}$ and $\alpha_{(2,1)}$.

We have proven in Theorem \ref{them:H^n} that:
       \[
    \begin{aligned}
     -(n+m-3)\alpha_{(n,m,0)}&{}=-\frac{n(n-1)}{n+m-1}\alpha_{(n+m-1,1,0)} -\frac{n(n-2)(m-1)}{n+m-2}\alpha_{(n+m-1,1,0)}\\
     &-\frac{m(n-1)(m-2)}{n+m-2}\alpha_{(n+m-1,1,0)},
    \end{aligned}
    \]
    \[
    \begin{aligned}
      (n+m+p-4)\alpha_{(n,m,p)}&=-\frac{n(n-1)}{n+m-1}\alpha_{(n+m-1,p,1)}+\frac{(n-1)mp}{m+p-1}\alpha_{(n,m+p-1,1)}\\
      &+\frac{m(m-1)}{m+p-1}\alpha_{(n,m+p-1,1)}-\frac{n(n-2)(m-1)}{n+m-2}\alpha_{(n+m-1,p,1)}\\
      &-\frac{m(n-1)(m-2)}{n+m-2}\alpha_{(n+m-1,p,1)}+\frac{m(m-2)(p-1)}{m+p-2}\alpha_{(n,m+p-1,1)}\\
      &-\frac{p(n-1)(m-1)}{n+m-1}\alpha_{(n+m,p-1,1)}+\frac{p(m-1)(p-2)}{m+p-2}\alpha_{(n,m+p-1,1)}\\
      &-\frac{(n-1)(m-1)}{n+m-1}\alpha_{(n+m,p,0)}+\frac{(m-1)(p-1)}{m+p-1}\alpha_{(n,m+p,0)}.
    \end{aligned}
    \]
    
     Choose 
\[
\varphi_1\in \tilde C^2,\quad \psi_1 \in \tilde C^3 = \Hom_\Lambda (\mathrm A_3, M)
\]
such that
\begin{align*}
  (\Delta^2\varphi_1-D^3\psi_1)(n,m,1)&=-\frac{n(n-1)}{n+m-1}\beta_{(n+m-1,1)}
    +(n+m-3) \beta_{(n,m)}\\
 &-\frac{n(n-2)(m-1)}{n+m-2}\beta_{(n+m-1,1)}-\frac{m(n-1)(m-2)}{n+m-2}\beta_{(n+m-1,1)}\\
  &-\frac{(n-1)(m-1)}{n+m-1}\beta_{(n+m,0)},\\
  (\Delta^2\varphi_1-D^3\psi_1)(n,1,0)&=-(n-2)\beta_{(n,0)},
   \end{align*}
   where 
   \[\beta_{(n,0)}=
    \begin{cases}
    -\frac{\alpha_{(n,1,0)}}{n-2}; & n\ge3, \\
   0; & \text{otherwise}.
    \end{cases}
    \]

    Then
    \[
    (\Delta^2\varphi_1-D^3\psi_1)(n,1,0)=\alpha_{(n,1,0)}=\varphi_{(n,1,0)};\quad n\ne2.
    \]
    
    Let us repeat the construction
with new values of $\beta$'s to get $\varphi'_1$ and $\psi'_1$:
     \[\beta_{(n,m)}=
    \begin{cases}
    \frac{\alpha_{(n,m,1)}}{n+m-3}; & n,m\ge2,\\
    0; & \text{otherwise},
    \end{cases}
    \]
    \begin{align*}
    (\Delta^2\varphi'_1-D^3\psi'_1)(n,1,0)&=0;\quad n\ne2,\\
    (\Delta^2\varphi'_1-D^3\psi'_1)(n,m,1)&=\alpha_{(n,m,1)}=\varphi_{(n,m,1)}.
    \end{align*}
    
Hence, 
$\varphi -\Delta^2\varphi_1 \in D^3\tilde \C^3$, so 
every cocycle in $\C^3$ is a coboundary or zero  except $\alpha_{(2,1,0)}$.

We have proven in Theorem \ref{them:H^n} that:
     for all $n,m,p,q\ge 2$ 
    \[
    \begin{aligned}
      -(n+m+p+q-5)\alpha_{(n,m,p,q)}&=
      -\frac{n(n-1)}{n+m-1}\alpha_{(n+m-1,p,q,1)}-\frac{(p-1)(q-1)}{p+q-1}\alpha_{(n,m,p+q,0)}\\
     &+\frac{m(m-1)}{m+p-1}\alpha_{(n,m+p-1,q,1)}-\frac{p(p-1)}{p+q-1}\alpha_{(n,m,p+q-1,1)} \\
      &+\frac{(n-1)mp}{m+p-1}\alpha_{(n,m+p-1,q,1)}-\frac{(n-1)pq}{p+q-1}\alpha_{(n,m,p+q-1,q,1)}\\
      &-\frac{(m-1)pq}{p+q-1}\alpha_{(n,m,p+q-1,q,1)}-\frac{n(n-2)(m-1)}{n+m-2}\alpha_{(n+m-1,p,q,1)}\\
      &-\frac{m(n-1)(m-2)}{n+m-2}\alpha_{(n+m-1,p,q,1)}\\
      &+\frac{m(m-2)(p-1)}{m+p-2}\alpha_{(n,m+p-1,q,1)}\\
      &-\frac{p(n-1)(m-1)}{n+m-1}\alpha_{(n+m,p-1,q,1)}\\
      &+\frac{p(m-1)(p-2)}{m+p-2}\alpha_{(n,m+p-1,q,1)}\\
      &-\frac{p(p-2)(q-1)}{p+q-2}\alpha_{(n,m,p+q-1,1)}\\
      &-\frac{q(n-1)(m-1)}{n+m-1}\alpha_{(n+m,p,q-1,1)}\\
      &+\frac{q(m-1)(p-1)}{m+p-1}\alpha_{(n,m+p,q-1,1)}\\
      &-\frac{q(p-1)(q-2)}{p+q-2}\alpha_{(n,m,p+q-1,1)}\\
      &-\frac{r(n-1)(m-1)}{n+m-1}\alpha_{(n+m,p,q,0)}\\
      &+\frac{(m-1)(p-1)}{m+p-1}\alpha_{(n,m+p,q,0)},
    \end{aligned}
    \]
    \begin{align*}
     (n+m+p-4)\alpha_{(n,m,p,0)}&=
      -\frac{n(n-1)}{n+m-1}\alpha_{(n+m-1,p,1,0)}+\frac{m(m-1)}{m+p-1}\alpha_{(n,m+p-1,1,0)} \\
      &-\frac{n(n-2)(m-1)}{n+m-2}\alpha_{(n+m-1,p,1,0)}+\frac{(n-1)mp}{m+p-1}\alpha_{(n,m+p-1,1,0)}\\
      \end{align*}

      \begin{align*}
      &-\frac{m(n-1)(m-2)}{n+m-2}\alpha_{(n+m-1,p,1,0)}+\frac{m(m-2)(p-1)}{m+p-2}\alpha_{(n,m+p-1,1,0)}\\
      &-\frac{p(n-1)(m-1)}{n+m-1}\alpha_{(n+m,p-1,1,0)}+\frac{p(m-1)(p-2)}{m+p-2}\alpha_{(n,m+p-1,1,0)}.
    \end{align*}
    
Let $\varphi_1\in \tilde C^3$ and $\psi_1\in \tilde C^4$ such that 
     \begin{align*}
     \varphi_1[v(n)v(m)v(p)]&=\beta_{(n,m,p)};\quad n\ge2, m\ge1, p\ge0,\\
      \psi_1[v(n)v(m)v(p)v(q)]&=\frac{(n-1)(m-1)}{n+m-1}\beta_{(n+m,p,q)}\\
      &-\frac{(m-1)(p-1)}{m+p-1}\beta_{(n,m+p,q)}\\
      &+\frac{(p-1)(q-1)}{p+q-1}\beta_{(n,m,p+q)};\quad n,m,p\ge2,q\ge 0,\\
      \psi_1[v(n)v(m)v(1)v(0)]&=\frac{(n-1)(m-1)}{n+m-1}\beta_{(n+m,1,0)}-\beta_{(n,m,1)};\quad n,m\ge2.
    \end{align*}
       where 
   \[\beta_{(n,m,0)}=
    \begin{cases}
    \frac{\alpha_{(n,m,1,0)}}{n+m-3}, & n,m\ge2, \\
   0; & \text{otherwise}.
    \end{cases}
    \]
    
    Then
    \[
    (\Delta^3\varphi_1-D^4\psi_1)(n,m,1,0)=\alpha_{(n,m,1,0)}=\varphi_{(n,m,1,0)};\quad n,m\ge2.
    \]
    
    Let us repeat the construction
with new values of $\beta$'s to get $\varphi'_1$ and $\psi'_1$:
     \[\beta_{(n,m,p)}=
    \begin{cases}
    -\frac{\alpha_{(n,m,p,1)}}{n+m+p-4}; & n,m,p\ge2,\\
    0; & \text{otherwise},
    \end{cases}
    \]
    \begin{align*}
    (\Delta^3\varphi'_1-D^4\psi'_1)(n,m,1,0)&=0;\quad n,m\ge2,\\
    (\Delta^3\varphi'_1-D^4\psi'_1)(n,m,p,1)&=\alpha_{(n,m,p,1)}=\varphi_{(n,m,p,1)}.
    \end{align*}
    
Hence $\varphi -\Delta^3\varphi_1 \in D^4\tilde \C^4$, so every cocycle in $\C^4$ is a coboundary.

We have proven in Theorem \ref{them:H^n} that: for all $i_1,\ldots,i_n\ge 2$, $i_{n+1}\ge0$,
 \begin{align*}
(-1)^{n+1} (i_1+\ldots+i_n-n-1 )\alpha_{(i_1,i_2,\ldots,i_n)}&=
   \sum^{n-1}_{j=1}(-1)^{j}\frac{i_j(i_j-1)}{i_j+i_{j+1}-1}\alpha_{(i_1,\ldots,i_j+i_{j+1}-1,\ldots,i_n,1)}\\
  &+\sum^n_{j=2}\sum^{j-1}_{t=1}(-1)^{j}\frac{i_ji_{j+1}}{i_j+i_{j+1}-1}(i_t-1)\\
&\times\alpha_{(i_1,\ldots,i_j+i_{j+1}-1,\ldots,i_n,1)}\\
\end{align*}
\begin{align*}
   &+\sum^{n+1}_{t=3}\sum^{t-2}_{j=1}(-1)^{j}i_t\frac{(i_j-1)(i_{j+1}-1)}{i_j+i_{j+1}-1}\alpha_{(i_1,\ldots ,i_j+i_{j+1},\ldots, i_t-1,\ldots ,i_n,1)}\\
     &+\sum^n_{t=1}(-1)^{t}i_t\frac{(i_t-2)(i_{t+1}-1)}{i_t+i_{t+1}-2}\alpha_{(i_1,\ldots ,i_t+i_{t+1}-1,\ldots ,i_n,1)}\\
     &+\sum^{n}_{t=1}(-1)^{t}i_{t+1}\frac{(i_t-1)(i_{t+1}-2)}{i_t+i_{t+1}-2}\alpha_{(i_1,\ldots ,i_t+i_{t+1}-1,\ldots ,i_n,1)},\\
 \end{align*}
  \[
 \begin{aligned}
(-1)^n (i_1+\ldots+i_{n-1}-n )\alpha_{(i_1,i_2,\ldots,i_{n-1},0)}&=
  \sum^{n-2}_{j=1}(-1)^{j}\frac{i_j(i_j-1)}{i_j+i_{j+1}-1}\alpha_{(i_1,\ldots,i_j+i_{j+1}-1,\ldots,i_{n-1},1,0)}\\
  &+\sum^{n-2}_{j=1}\sum^{j-1}_{t=1}(-1)^{j}\frac{i_ji_{j+1}}{i_j+i_{j+1}-1}(i_t-1)\\
  &\times\alpha_{(i_1,\ldots,i_j+i_{j+1}-1,\ldots,i_{n-1},1,0)}\\
   &+\sum^{n-1}_{t=3}\sum^{t-2}_{j=1}(-1)^{j}i_t\frac{(i_j-1)(i_{j+1}-1)}{i_j+i_{j+1}-1}\\
   &\times\alpha_{(i_1,\ldots ,i_j+i_{j+1},\ldots, i_t-1,\ldots ,i_{n-1},1,0)}\\
    &+\sum^{n-2}_{t=1}(-1)^{t}i_t\frac{(i_t-2)(i_{t+1}-1)}{i_t+i_{t+1}-2}\\
    &\times\alpha_{(i_1,\ldots ,i_t+i_{t+1}-1,\ldots ,i_{n-1},1,0)}\\
     &+\sum^{n-2}_{t=1}(-1)^{t}i_{t+1}\frac{(i_t-1)(i_{t+1}-2)}{i_t+i_{t+1}-2}\\
     &\times\alpha_{(i_1,\ldots ,i_t+i_{t+1}-1,\ldots ,i_{n-1},1,0)}.
 \end{aligned}
 \]
 
 Let $\varphi_1\in \tilde C^{n-1}, \psi_1\in \tilde C^n$, such that
     \[
     \varphi_1(i_1,i_2,\ldots,i_{n-1})=\beta_{(i_1,i_2,\ldots,i_{n-1})},
     \]
     \begin{align*}
     \psi_1(i_1,i_2,\ldots,i_{n})&=\sum^{n-1}_{j=1}(-1)^{j}\frac{(i_j-1)(i_{j+1}-1)}{i_j+i_{j+1}-1}\beta_{(i_1,\ldots,i_j+i_{j+1},\ldots,i_n)}\\
     \psi_1(i_1,i_2,\ldots,i_{n-2},1,0)&=\sum^{n-2}_{j=1}(-1)^{j}\frac{(i_j-1)(i_{j+1}-1)}{i_j+i_{j+1}-1}\beta_{(i_1,\ldots,i_j+i_{j+1},\ldots,i_{n-2},1,0)}\\
     & +(-1)^{n-1}\beta_{(i_1,i_2,\ldots,i_{n-1},1)},
     \end{align*}
where 
   \[\beta_{(i_1,i_2,\ldots,i_{n-2},0)}=
    \begin{cases}
    \frac{(-1)^n\alpha_{(i_1,i_2,\ldots,i_{n-2},1,0)}}{i_1+\ldots+i_{n-2}-n+1};& i_1,\ldots,i_{n-2}\ge2,\\
   0; & \text{otherwise}.
    \end{cases}
    \]
    
    Then
    \[
   (\Delta^{n-1}\varphi_1-D^n\psi_1)(i_1,i_2,\ldots,i_{n-2},1,0)=\varphi(i_1,i_2,\ldots,i_{n-2},1,0);\quad i_1,\ldots,i_{n-2}\ge2.
   \]
   
 Let us repeat the construction
with new values of $\beta$'s to get $\varphi'_1$ and $\psi'_1$:
     \[\beta_{(i_1,\ldots,i_{n-1})}=
    \begin{cases}
   -\frac{(-1)^n\alpha_{(i_1,\ldots,i_{n-1},1)}}{i_1+\ldots+i_{n-1}-n};& i_1,\ldots,i_{n-1}\ge2,\\
    0; & \text{otherwise}.
    \end{cases}
    \]
 \begin{align*}
   (\Delta^{n-1}\varphi'_1-D^n\psi'_1)(i_1,i_2,\ldots,i_{n-2},1,0)&=0;\quad i_1,\ldots,i_{n-2}\ge2,\\
   (\Delta^{n-1}\varphi'_1-D^n\psi'_1)(i_1,i_2,\ldots,i_{n-2},i_{n-1},1)&=\varphi(i_1,i_2,\ldots,i_{n-2},i_{n-1},1);\quad i_1,\ldots,i_{n-1}\ge2.
   \end{align*}
   
    Hence $\varphi -\Delta^{n-1}\varphi_1 \in D^n\tilde \C^n$, so 
every cocycle in $\C^n$ is a coboundary.
  \end{proof} 
  
  \begin{corollary}\label{cor:AllFinite}
Let $M$ be a finite module over $U(3)$. Then $\Homol^k(U(3),M)=0$ for all $k\ge 4$.
\end{corollary}

\begin{proof}
Let $M$ be a finite conformal module over $U(3)$.
Then, in particular, $M$ is a finite module over the Virasoro Lie conformal algebra $\Vir $. 
Hence there exists a chain of $\Vir$-submodules
 (see, e.g., \cite[Lemma 3.3]{kolesnikov})
 \[
 0= M_{-1} \subset M_0 \subset \ldots \subset M_n=M,
 \]
where $M_i/M_{i-1}$,  $i=0,\ldots,n$, is 
either isomorphic to a $\Vir $-module $M_{(\alpha,\Delta)}$,
or trivial torsion-free module $\Bbbk [\partial ]u$ with 
$(v\oo\lambda u)=0$, 
or coincides with its torsion (hence, trivial).
Note that a $\Vir$-submodule of an $U(3)$-module $M$ is itself a 
$U(3)$-module, therefore, all $M_i$ are $U(3)$-modules and 
so are $M_i/M_{i-1}$. Hence, all irreducible quotients 
are of type $M_{(\alpha, \Delta)}$.

The case of torsion module was considered in \cite{Akl}. 
There is no difference between the scalar $U(3)$-module 
$\Bbbk $ 
and the trivial torsion-free module $M=\Bbbk [\partial ]u$. 
One may also apply the technique 
of Theorem~\ref{them:H^n} to the case of the trivial module $M$ as above 
to prove that $\Homol^k(U(3),M)=0$ for all $k\ge 4$.
Finally, both Theorem \ref{thm:4.3} and Theorem \ref{them:H^n} imply
that
 \[
\Homol^k(U(3), M_i/M_{i-1} )=0, \quad i=0,\dots, n, 
\]
for all $k\ge 4$.
The short exact sequence of modules 
\[
0\to M_{i-1} \to M_i \to 
M_i/M_{i-1} \to 0
\]
leads to the long exact sequence of cohomology groups
\[
\begin{aligned}
\dots & \to \Homol^k(U(3),M_{i-1})
        \to \Homol^k(U(3),M_i)
        \to \Homol^k (U(3), {M_i/M_{i-1}}) \\
& \to \Homol^{k+1}(U(3),M_{i-1}) 
  \to \Homol^{k+1}(U(3),M_i)
        \to \Homol^{k+1} (U(3), {M_i/M_{i-1}}) \\
&        \to \dots 
\end{aligned}
\]
for every $i=1,\ldots,n$. 
Since  
\[
\Homol^k(U(3),M_0)=0,\quad \Homol^k(U(3),{M_1/M_0})=0
\]
for all  $k\ge4$, we obtain $\Homol^k(U(3), M_1)=0$.
Proceed by induction on $i=1,\dots, n$ to obtain
$\Homol^k(U(3),M_n)=0$, for all $k\ge4$.
\end{proof}

\subsection*{Acknowledgments}
The work was supported by Russian Science Foundation, Project 23-21-00504.






\EditInfo{November 5, 2024}{November 8, 2024}{David Towers and Ivan Kaygorodov}

\end{document}